\documentclass{amsart}
\usepackage{amsmath,amsthm,amssymb}
\usepackage{hyperref}
\usepackage{pstricks,pst-node,multido}

\usepackage{a4wide}

\newcommand\PT{\mathcal{PT}}
\newcommand\diag{\operatorname{diag}}
\newcommand\flr[1]{\left\lfloor #1\right\rfloor}

\newcommand\row{\operatorname{row}}
\newcommand\des{\operatorname{des}}

\newcommand\hasc{\operatorname{hasc}}
\newcommand\asc{\operatorname{asc}}
\newcommand\pat{\operatorname{pat}}
\newcommand\fwex{\operatorname{fwex}}
\newcommand\fneg{\operatorname{fneg}}
\newcommand\sh{\operatorname{sh}}
\newcommand\w{\operatorname{w}}
\newcommand\tr{\operatorname{tr}}

\newcommand\fdes{\operatorname{fdes}}
\newcommand\fexc{\operatorname{fexc}}

\renewcommand\neg{\operatorname{neg}}

\newcommand\wex{\operatorname{wex}}
\newcommand\exc{\operatorname{exc}}

\newcommand\st{\operatorname{st}}
\newcommand\rev{\mathbf{r}}
\newcommand{\cmo}{\genfrac{}{}{0pt}{}{}{-}}
\newcommand{\cmp}{\genfrac{}{}{0pt}{}{}{+}}

\newcommand\cro{\operatorname{cr}}
\newcommand\al{\operatorname{al}}
\newcommand\so{\operatorname{so}}
\newcommand\floor[1]{\left\lfloor #1 \right\rfloor}
\psset{unit=15pt}
\newcount\ax \newcount\ay
\newcount\bx \newcount\by
\newcount\cx \newcount\cy
\newcount\dx \newcount\dy
\def\cell(#1,#2)[#3]{
\ax=#2 \ay=#1
\multiply\ay by-1
\bx=\ax \by=\ay
\cx=\ax \cy=\ay
\dx=\ax \dy=\ay
\advance\bx by-1
\advance\dy by1
\advance\cx by-1
\advance\cy by1
\psline (\dx,\dy)(\ax,\ay)(\bx,\by)
\rput(\number\cx.5,
\ifnum\cy=0 -0.5\else\number\cy.5\fi){#3}
}
\def\vvput#1#2{\pnode(#1,1){#2} \pscircle*(#1,1){.1}}

\def\vput#1{\pnode(#1,1){#1} \pscircle*(#1,1){.1}}
\def\edge#1#2{\ncarc[arcangle=80]{c-c}{#1}{#2}}

\def\pig#1{\rput(#1,1){\psellipse(0,0)(.5,.25)}}
\def\halfarc(#1){
\raisebox{-7pt}{\pspicture(-.5,-.5)(.5,.5)\psdot(0,0)\psline(0,0)(#1)\endpspicture}}

\newtheorem{thm}{Theorem}[section]
\newtheorem{lem}[thm]{Lemma}
\newtheorem{prop}[thm]{Proposition}

\theoremstyle{definition}
\newtheorem{example}{Example}
\newtheorem{defn}{Definition}

\theoremstyle{remark}
\newtheorem{remark}{Remark}

\newcommand{\ket}[1]{\ensuremath{|#1\rangle}}
\newcommand{\bra}[1]{\ensuremath{\langle #1|}}

\DeclareMathOperator{\mot}{31-2}
\DeclareMathOperator{\mott}{2-31}

\newcommand{\nodr}{ \psset{unit=2mm}\begin{pspicture}(-1,-0.6)(1.4,1.6)
\psdot(0,0)
\pscurve(0,0)(0.7,0.8)(1,1)\pscurve[linestyle=dotted,dotsep=0.2mm](1,1)(1.3,1.2)(1.6,1.3)
\pscurve(0,0)(0.7,-0.8)(1,-1)\pscurve[linestyle=dotted,dotsep=0.2mm](1,-1)(1.3,-1.2)(1.6,-1.3)
\end{pspicture} }
\newcommand{\nodhp}{\psset{unit=2mm}\begin{pspicture}(-1,-0.6)(1.4,2)
 \psdot(0,0)
\pscurve(0,0)(0.7,0.8)(1,1)\pscurve[linestyle=dotted,dotsep=0.2mm](1,1)(1.3,1.2)(1.6,1.3)
\pscurve(0,0)(-0.7,0.8)(-1,1)\pscurve[linestyle=dotted,dotsep=0.2mm](-1,1)(-1.3,1.2)(-1.6,1.3)
\rput(0.6,1.4){+}
\end{pspicture}}
\newcommand{\nodhm}{\psset{unit=2mm}\begin{pspicture}(-1,-0.6)(1.4,2)
 \psdot(0,0)
\pscurve(0,0)(0.7,0.8)(1,1)\pscurve[linestyle=dotted,dotsep=0.2mm](1,1)(1.3,1.2)(1.6,1.3)
\pscurve(0,0)(-0.7,0.8)(-1,1)\pscurve[linestyle=dotted,dotsep=0.2mm](-1,1)(-1.3,1.2)(-1.6,1.3)
\rput(0.6,1.6){$-$}
\end{pspicture}}
\newcommand{\nodbp}{\psset{unit=2mm}\begin{pspicture}(-1,-0.6)(1.4,1.6)
 \psdot(0,0)
\pscurve(0,0)(0.7,-0.8)(1,-1)\pscurve[linestyle=dotted,dotsep=0.2mm](1,-1)(1.3,-1.2)(1.6,-1.3)
\pscurve(0,0)(-0.7,-0.8)(-1,-1)\pscurve[linestyle=dotted,dotsep=0.2mm](-1,-1)(-1.3,-1.2)(-1.6,-1.3)
\rput(0.6,-1.6){$+$}
\end{pspicture}}
\newcommand{\nodbm}{\psset{unit=2mm}\begin{pspicture}(-1,-0.6)(1.4,1.6)
 \psdot(0,0)
\pscurve(0,0)(0.7,-0.8)(1,-1)\pscurve[linestyle=dotted,dotsep=0.2mm](1,-1)(1.3,-1.2)(1.6,-1.3)
\pscurve(0,0)(-0.7,-0.8)(-1,-1)\pscurve[linestyle=dotted,dotsep=0.2mm](-1,-1)(-1.3,-1.2)(-1.6,-1.3)
\rput(0.6,-1.6){$-$}
\end{pspicture}}
\newcommand{\nodg}{\psset{unit=2mm}\begin{pspicture}(-1,-0.7)(1.4,1.3)
 \psdot(0,0)
\pscurve(0,0)(-0.7,0.8)(-1,1)\pscurve[linestyle=dotted,dotsep=0.2mm](-1,1)(-1.3,1.2)(-1.6,1.3)
\pscurve(0,0)(-0.7,-0.8)(-1,-1)\pscurve[linestyle=dotted,dotsep=0.2mm](-1,-1)(-1.3,-1.2)(-1.6,-1.3)
\end{pspicture} }
\newcommand{\nodfm}{ \psset{unit=2mm}\begin{pspicture}(-1,-0.6)(1.4,1.6)
 \psdot(0,0)\pscurve(0,0)(-0.3,0.6)(0,0.8)(0.3,0.6)(0,0) \rput(0,1.2){$-$}\end{pspicture}
}
\newcommand{\nodfp}{ \psset{unit=2mm}\begin{pspicture}(-1,-0.6)(1.4,1.6)
 \psdot(0,0)\pscurve(0,0)(-0.3,0.6)(0,0.8)(0.3,0.6)(0,0) \rput(0,1.4){$+$}\end{pspicture}
}

\title{Combinatorics of the permutation tableaux of type $B$}

\author{Sylvie Corteel, Matthieu Josuat-Verg\`es and Jang Soo Kim}

\address{LIAFA, CNRS and Universit\'e Paris-Diderot, Case 7014, 75205 Paris Cedex 13, FRANCE}

\email{corteel@liafa.jussieu.fr}

\address{Institut Gaspard Monge, CNRS and Universit\'e de Marne-la-Vall\'ee, FRANCE}

\email{matthieu.josuat-verges@univ-mlv.fr}

\address{
School of Mathematics,
University of Minnesota,
Minneapolis, Minnesota 55455, USA}

\email{kimjs@math.umn.edu}

\thanks{All authors are partially supported by the grant ANR08-JCJC-0011.
The second author was supported by the Austrian Science foundation FWF (START grant Y463)
while he was postdoctoral researcher at the university of Vienna.
}

\begin{document} 

\begin{abstract}
Permutation tableaux are combinatorial objects related with permutations and various statistics on them.
They appeared in connection with total positivity in Grassmannians, and stationary probabilities in a PASEP model.
In particular they gave rise to an interesting $q$-analog of Eulerian numbers. The purpose of this
article is to study some combinatorial properties of type $B$ permutation tableaux, defined by Lam and Williams,
and links with signed permutation statistics.

We show that many of the tools used for permutation tableaux generalize in this case, including: the Matrix Ansatz 
(a method originally related with the PASEP), bijections with labeled paths and links with continued fractions, bijections 
with signed permutations. In particular we obtain a $q$-analog of the type $B$ Eulerian numbers, having a lot in common
with the previously known $q$-Eulerian numbers: for example they have a nice symmetry property, they have the type
$B$ Narayana numbers as constant terms.

The signed permutation statistics arising here are of several kinds. Firstly, there are several variants of descents
and excedances, and more precisely of flag descents and flag excedances. 
Other statistics are the crossings and alignments,
which generalize a previous definition on (unsigned) permutations. There are also some pattern-like statistics arising
from variants of the bijection of Fran\c con and Viennot.
\end{abstract}

\setcounter{tocdepth}{1}
\maketitle
\tableofcontents

\section{Introduction}

Permutation tableaux of type $B$ have been introduced in \cite{LamWilliams}, and 
further studied in \cite{CJW,CK}. We start this introduction by some definitions.
A \emph{Ferrers diagram} is a top and left justified arrangement of square cells with
possibly empty rows and columns.  The \emph{length} of a Ferrers diagram is the
sum of the number of rows and the number of columns. If a Ferrers diagram is of
length $n$, we label the steps in the south-east border of the Ferrers diagram
with $1,2,\ldots,n$ from north-east to south-west. We will also call the step
labeled with $i$ the $i$th step. The following is an example of a Ferrers
diagram.
\begin{equation}
  \label{eq:ferrers}
\raisebox{-20pt}{
  \begin{pspicture}(0,0)(4,-3)
    \psline(4,0)(0,0)(0,-3)
    \cell(1,1)[]  \cell(1,2)[]  \cell(1,3)[]
    \cell(2,1)[]  \cell(2,2)[]
    \tiny    \rput(3.5,-.2){1}
    \rput(2.5,-1.2){3}
    \rput(1.5,-2.2){5}
    \rput(0.5,-2.2){6}
    \rput(0.2,-2.5){7}
    \rput(2.2,-1.5){4}
    \rput(3.2,-0.5){2}
 \end{pspicture}
}\end{equation}

A permutation tableau is a 0,1-filling of a Ferrers diagram satisfying the
following conditions: (1) each column has at least one $1$ and (2) there is no
$0$ which has a $1$ above it in the same column and a $1$ to the left of it in
the same row. The following is an example of a permutation tableau. 
\begin{equation}
  \label{eq:pt}
\raisebox{-30pt}{
\begin{pspicture}(0,0)(4,-4) 
\psline(0,-4)(0,0)(4,0) \cell(1,1)[1]
  \cell(1,2)[0] \cell(1,3)[0] \cell(1,4)[1] \cell(2,1)[0] \cell(2,2)[1]
  \cell(2,3)[1] \cell(3,1)[0] \cell(3,2)[0] \cell(3,3)[1] \tiny
  \rput(4.2,-0.5){1} \rput(3.2,-1.5){3} \rput(3.2,-2.5){4} \rput(0.2,-3.5){8}
  \rput(3.5,-1.2){2} \rput(2.5,-3.2){5} \rput(1.5,-3.2){6} \rput(0.5,-3.2){7}
\end{pspicture}}
\end{equation}

For a Ferrers diagram $F$ with $k$ columns, the \emph{shifted} Ferrers diagram
of $F$, denoted $\overline F$, is the diagram obtained from $F$ by adding $k$
rows of size $1,2,\ldots,k$ above it in increasing order. The \emph{length} of
$\overline F$ is defined to be the length of $F$.  A \emph{diagonal cell} is the
rightmost cell of an added row. For example, if $F$ is the Ferrers diagram in
\eqref{eq:ferrers} then $\overline F$ is the diagram in \eqref{eq:ptb} without
the 0's and 1's.

A \emph{type~$B$ permutation tableau} of length $n$ is a $0,1$-filling of a
shifted Ferrers diagram of length $n$ satisfying the following conditions: (1)
each column has at least one $1$, (2) there is no $0$ which has a $1$ above it
in the same column and a $1$ to the left of it in the same row, and (3) if a $0$
is in a diagonal cell, then it does not have a $1$ to the left of it in the same
row.  The following is an example of a type~$B$ permutation tableau. 
\begin{equation}
  \label{eq:ptb}
\raisebox{-50pt}{
  \begin{pspicture}(0,0)(4,-7)
    \multido{\n=0+1}{4}{\rput(\n,-\n){\psline{C-C}(1,0)}} \psline(0,0)(0,-7)
    \cell(1,1)[0] \cell(2,1)[1] \cell(2,2)[1] \cell(3,1)[0] \cell(3,2)[0]
    \cell(3,3)[0] \cell(4,1)[0] \cell(4,2)[1] \cell(4,3)[0] \cell(4,4)[1]
    \cell(5,1)[1] \cell(5,2)[1] \cell(5,3)[1] \cell(6,1)[0] \cell(6,2)[1]
    \tiny
    \rput(0,-4){
      \tiny    \rput(3.5,-.2){1}
      \rput(2.5,-1.2){3}
      \rput(1.5,-2.2){5}
      \rput(0.5,-2.2){6}
      \rput(0.2,-2.5){7}
      \rput(2.2,-1.5){4}
      \rput(3.2,-0.5){2}
    }
\end{pspicture}}
\end{equation}

We respectively denote by $\PT(n)$ and $\PT_B(n)$ the sets of permutation
tableaux and type $B$ permutation tableaux of length $n$. For $T\in \PT(n)$ or
$T\in\PT_B(n)$, a $1$ is called \emph{superfluous} if it has a $1$ above it in
the same column. Let $\so(T)$ denote the number of superfluous ones in $T$ and
$\row(T)$ denote the number of rows of $T$ except the added rows. For
$T\in\PT_B(n)$, let $\diag(T)$ be the number of ones on the diagonal of the
added rows. Let
\[
B_n(y,t,q)=\sum_{T\in  \PT_B(n)}y^{2\row(T)+\diag(T)}t^{\diag(T)}q^{\so(T)}.
\]
For example:
\begin{align*}
  B_0(y,t,q) &= 1, \\
  B_1(y,t,q) &= y^2+yt, \\
  B_2(y,t,q) &= y^4+(2t+tq)y^3+(t^2q+t^2+1)y^2+ty.
\end{align*}
Let also $B_{n,k}(t,q)=[y^k]B_n(y,t,q)$, which is nonzero when $n=k=0$ or $1\leq
k \leq 2n$. Here $[y^k]f(y)$ means the coefficient of $y^k$ in $f(y)$. 

In the case $t=0$, we have $B_{n,2k+1}(0,q)=0$, and $B_{n,2k}(0,q)$ is the $q$-Eulerian number $E_{n,k}(q)$ 
defined by Williams \cite{Williams2005} such that:
\[
  E_{n,k}(q) =  \sum_{\substack{T\in  \PT(n) \\ \row(T)=k}} q^{\so(T)}.
\]
Indeed by condition (3) in the definition, a type $B$ permutation tableau with no $1$ in a diagonal cell
has only $0$'s in the added row and is equivalent to a permutation tableau. Some properties are following:
\begin{itemize}
 \item There is a symmetry such that $E_{n,k}(q)=E_{n,n+1-k}(q)$, and some special values are: $E_{n,k}(-1)=\binom{n-1}{k-1}$, 
$E_{n,k}(0)=\frac 1n \binom{n}{k}\binom{n}{k-1}$ (the Narayana number) \cite{Williams2005}. 
\item There is a continued fraction expansion for the generating function
  $\sum_{k=0}^n E_{n,k}(q) y^kz^n$, and two combinatorial
  interpretation in terms of permutations are known: a first one with excedances
  and crossings, a second one with descents and the pattern 31-2
  \cite{Corteel2007}.
\item There is an explicit formula for the $q$-Eulerian polynomial $\sum_{k=1}^n
  y^k E_{n,k}(q) $ \cite{J,Williams2005}.
\end{itemize}
We will see throughout the article that all these properties can be generalized
to $B_{n,k}(t,q)$ or $B_{n,k}^*(t,q)$ which is a variation of $B_{n,k}(t,q)$.

However the link with type $B$ Eulerian numbers is not completely
immediate. Adin et al.~\cite{adin} defined $\fdes(\pi)$, the number of flag descents of $\pi\in B_n$.
This refines $\des_B(\pi)$ the Coxeter theoretic definition of the number of descents (see
\cite{bjorner}) in the sense that $ \lceil \fdes(\pi)/2 \rceil =
\des_B(\pi)$. Here $B_n$ denotes the set of signed permutations of $[n]$, see
Section~\ref{bij} for the definition of the statistics. There is also a definition of flag
excedances \cite{foata}.  These flag statistics are the relevant ones in this
context.

%

Using a zigzag bijection of \cite{LamWilliams}, we will show in Section
\ref{bij} the following theorem.
\begin{thm} \label{thm:ZZ} For a nonnegative integer $n$ we have
\begin{equation} \label{bnfwexcro}
B_n(y,t,q)=\sum_{\pi\in B_n} y^{\fwex(\pi)}t^{\neg(\pi)}q^{\cro(\pi)}.
\end{equation}
Equivalently, we have
\begin{equation}
  \label{eq:Bnk}
B_{n,k}(t,q)=\sum_{\substack{\pi\in B_n\\ \fwex(\pi)=k}} t^{\neg(\pi)}q^{\cro(\pi)}.  
\end{equation}
\end{thm}
Here, $\fwex(\pi)$ denotes the number of flag weak excedances, {\it i.e.} a
variant of flag excedances, equidistributed with $(\fdes+1)$. Theorem~\ref{thm:ZZ}
is a type $B$ analog of a result of Steingr\'{i}msson and Williams
\cite{Steingrimsson2007}. In Section~\ref{bij} we study how $\fwex$ is related
with the type $B$ descent $\des_B$. 

In Section~\ref{sec:qEuler} we define a $q$-Eulerian number of type $B$ by
\begin{equation}
  \label{eq:12}
E_{n,k}^B(q) =\sum_{\substack{\pi\in B_n\\ \flr{\fwex(\pi)/2}=k}} q^{\cro(\pi)},
\end{equation}
which is equivalent to $E^B_{n,k}(q)=B_{n,2k}(1,q)+B_{n,2k+1}(1,q)$.  We then
introduce a ``pignose diagram'' representing a signed permutation and prove the
following theorem.

\begin{thm} \label{qeulrianb}
 For any $0\leq k \leq n$, we have $E_{n,k}^B(-1)=\binom nk$,  $E_{n,k}^B(0)=\binom nk^2$ (the Narayana numbers of type $B$),
$E_{n,k}(1)$ is the type $B$ Eulerian number (see \cite{bjorner}), and 
\begin{equation}
  \label{eq:Esym}
 E^B_{n,k}(q)=E^B_{n,n-k}(q).  
\end{equation}
\end{thm}

In fact we will prove a more general symmetry than \eqref{eq:Esym}.  Observe
that the identity $B_{n,k}(1,q)=B_{n,2n+1-k}(1,q)$ implies
\eqref{eq:Esym}. However, for general $t$, we have $B_{n,k}(t,q) \ne
B_{n,2n+1-k}(t,q)$. For instance, $B_{1,1}(t,q)=t$ and $B_{1,2}(t,q)=1$.
There is a way to fix this discrepancy. Let
\[
B_{n,k}^*(t,q) =\sum_{\substack{\pi\in B_n\\ \fwex(\pi)=k}} t^{\neg(\pi) + \chi(\pi_1>0)}q^{\cro(\pi)},
\]
where $\chi(\pi_1>0)$ is $1$ if $\pi_1>0$ and $0$ otherwise. 
We will prove the following symmetry by a combinatorial argument:
\begin{thm} \label{bnsym}
For $1\le k\le 2n$, we have
\[ 
  B_{n,k}^*(t,q)=B_{n,2n+1-k}^*(t,q).
\]  
In particular, when $t=1$, we have $B_{n,k}(1,q)=B_{n,2n+1-k}(1,q)$. 
\end{thm}

Note that $E_{n,k}(q)$ and $E_{n,k}^B(q)$ are generating functions in which $q$
records the statistic $\cro(\pi)$. There is another statistic $\al(\pi)$
counting the number of alignments. Corteel \cite{Corteel2007} showed that the
two statistics $\cro(\pi)$ and $\al(\pi)$ in $S_n$ are closely related:
\begin{prop}\cite{Corteel2007}\label{prop:corteel_intro}
If $\sigma\in S_n$ has $k$ weak excedances, then
\[ \cro(\sigma) + \al(\sigma) = (k-1)(n-k).\]
\end{prop}
In Section~\ref{cro_ali} we prove the following proposition which is a type $B$
analog of the above result.
\begin{prop}\label{prop:corteelB_intro}
For $\pi\in B_n$ with $\fwex(\pi)=k$, we have
\[ 2\cro(\pi) + \al(\pi) = n^2-2n+k.\]
\end{prop}

Besides the zigzag bijection, another important result coming from permutation tableaux is the following:
\begin{thm}(Matrix Ansatz \cite{CJW})
Let $D$ and $E$ be matrices, $\langle W|$ a row vector, and $|V\rangle$ a column vector, such that:
\begin{eqnarray}  \label{ansatz}
DE = qED+D+E, \qquad 
D|V\rangle = |V\rangle, \qquad
\langle W|E = yt\langle W|D.
\end{eqnarray}
Then we have:
\[
B_n(y,t,q)=\langle W|(y^2D+E)^n|V \rangle.
\]
\end{thm}

This can be seen as an abstract rule to compute $B_n(y,t,q)$, but it is also
useful to have explicitly $D$, $E$, $\bra{W}$ and $\ket{V}$ satisfying the
relations ({\it i.e.} solutions of the Matrix Ansatz), see \cite{CJW}.  We will
give two such solutions in Section \ref{MA}. The following induction formula
appears in relation with one of the solutions. We use the standard notation
$[n]_q = (1-q^n)/(1-q)$.

\begin{thm} \label{bnrec} We have $B_0=1$, and 
\[
B_{n+1}(y,t,q)=(y+t)D_q[(1+yt)B_n(y,t,q)],
\]
where $D_q$ is the $q$-derivative with respect to $t$, {\it i.e.} the linear operator that sends $t^n$ to $[n]_qt^{n-1}$.
\end{thm}

As another consequence of the solutions of the Matrix Ansatz, we will see $B_n(y,t,q)$ as a generating function of 
labeled Motzkin paths or suffixes of labeled Motzkin paths. In particular, we obtain in this way a continued fraction:

\begin{thm} \label{bnfrac}
Let
$\gamma_h  = y^2[h+1]_q + [h]_q + tyq^h([h]_q+[h+1]_q)$ and 
$\lambda_h = y[h]_q^2(y+tq^{h-1})(1+ytq^h)$. Then we have
\[
  \sum_{n\geq0} B_n(y,t,q) z^n = 
  \frac{1}{1-\gamma_0z} \cmo \frac{\lambda_1z^2}{1-\gamma_1z} \cmo \frac{\lambda_2z^2}{1-\gamma_2z} \cmo \cdots .
\]
\end{thm}

We use the notation $\frac {a_1}{b_1} \cmo \frac {a_2}{b_2} \cmo \dots =  (a_1/(b_1-(a_2/(b_2- \dots )))) $
Note that this kind of continued fraction (called J-fraction) are related with moments of orthogonal polynomials \cite{Vie}.

We will show in Section~\ref{paths} that the two kinds of paths are in bijection
with signed permutations, using variants of classical bijections of Fran\c{c}on and
Viennot \cite{FV}, Foata and Zeilberger \cite{FZ}. We obtain two other interpretations
of $B_n(y,t,q)$ where $y$ follows a descent statistic and $q$ a pattern
statistic. Once again, this is a type $B$ analog of results on 
permutation tableaux and permutations \cite{CN,Steingrimsson2007}. The results are the following
(see Section~\ref{paths} for the definitions):
\begin{thm} \label{bndes} For $n\geq 1$, we have:
\begin{equation*}
 B_n(y,t,q) = \sum_{\pi \in B_n}  y^{\hasc(\pi)}    t^{\neg(\pi)} q^{ \pat (\pi) }
= \sum_{\pi \in B_n}  y^{\fdes(\pi) + 1 }    t^{\fneg(\pi)} q^{ \mot^+ (\pi) }.
\end{equation*}
\end{thm}
This could also be derived directly from the permutation tableaux using
the techniques developed in \cite{ABN2010}.

At last but not least we give an enumeration formula for $B_n(y,1,q)$, obtained from the continued fraction:
\begin{thm}\label{bnformula}
For a nonnegative integer $n$, we have
\[
B_n(y,1,q)=\frac{1}{(1-q)^n}\sum_{k=0}^n\left(\sum_{i=0}^{2n-2k}y^i
\tbinom n{k+\lceil i/2\rceil} \tbinom n{\lfloor i/2\rfloor}\right)
\left(\sum_{j=0}^{2k}y^{j+1}q^{(j+1)(2k-j)/2}\right).
\]
\end{thm}

This will be given in Section \ref{EF}, and the proof uses techniques developed
in \cite{JK}.  We conclude in Section \ref{conc} by a list of open problems.


\section{\texorpdfstring{Statistics of permutation tableaux of type $B$ and
    signed permutations}{Statistics of permutation tableaux of type B and signed permutations}}
\label{bij}

\subsection{The zigzag bijection}

Steingr\'imsson and Williams \cite{Steingrimsson2007} defined a bijection
$\Phi:\PT(n)\to S_n$ using zigzag paths.  For given $T\in\PT(n)$, the
corresponding permutation $\Phi(T)=\pi=\pi_1\cdots\pi_n\in S_n$ is obtained as
follows. Here $S_n$ denotes the set of permutations of
$[n]=\{1,2,\dots,n\}$. For each $i\in[n]$, we travel from the $i$th step of the
Ferrers diagram.  If the $i$th step is horizontal (resp.~vertical), then travel
to the topmost (resp.~leftmost) $1$ in the column (resp.~row) containing the
$i$th step, and then travel to the east (resp.~south) changing the direction to
south or east whenever we meet a $1$. Then $\pi_i$ is defined to be the label of
the step at which we finally arrive. If there is no $1$ in the row containing
the $i$th step, then $\pi_i=i$. For example, if $T$ is the permutation tableau
in \eqref{eq:pt}, then $\Phi(T)=7,1,6,5,3,4,2,8$.

The bijection $\Phi$ is naturally extended to the bijection $\Phi_B:\PT_B(n)\to
B_n$, where we denote by $B_n$ the set of signed permutations on $[n]$.  A
\emph{signed permutation} on $[n]$ is a permutation of $[n]$ where each integer
may be negated. For example, $\pi=4,2,-5,6,-1,3$ is a signed permutation of
$[6]$.  For $T\in\PT_B(n)$, we construct $\pi\in B_n$ in a similar way using
zigzag paths except the following. If the column containing the $i$th step has a
$1$ in the diagonal, then we travel from the $i$th step to the $1$ in the
diagonal, travel to the leftmost $1$ in the row containing the $1$ and then
travel to the south changing the direction to south or east whenever we meet a
$1$. If we arrive at the $j$th step then $\pi_i=-j$. For example, if $T$ is the
type $B$ permutation tableau in Figure~\ref{fig:ptb}, we have
$\Phi_B(T)=-3,6,2,5,-4,1,7$.

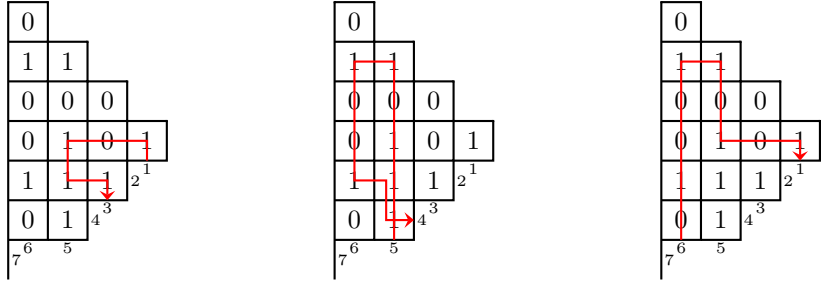
\begin{figure}
  \begin{pspicture}(0,0)(4,-7)
    \multido{\n=0+1}{4}{\rput(\n,-\n){\psline{C-C}(1,0)}} \psline(0,0)(0,-7)
    \cell(1,1)[0] \cell(2,1)[1] \cell(2,2)[1] \cell(3,1)[0] \cell(3,2)[0]
    \cell(3,3)[0] \cell(4,1)[0] \cell(4,2)[1] \cell(4,3)[0] \cell(4,4)[1]
    \cell(5,1)[1] \cell(5,2)[1] \cell(5,3)[1] \cell(6,1)[0] \cell(6,2)[1]
    \tiny
    \rput(0,-4){
      \tiny    \rput(3.5,-.2){1}
      \rput(2.5,-1.2){3}
      \rput(1.5,-2.2){5}
      \rput(0.5,-2.2){6}
      \rput(0.2,-2.5){7}
      \rput(2.2,-1.5){4}
      \rput(3.2,-0.5){2}
    }
  \psline[linecolor=red,arrowsize=.4,
  arrowlength=.6]{->}(3.5,-4)(3.5,-3.5)(1.5,-3.5)(1.5,-4.5)(2.5,-4.5)(2.5,-5)
\end{pspicture}
\hspace{2cm}
  \begin{pspicture}(0,0)(4,-7)
    \multido{\n=0+1}{4}{\rput(\n,-\n){\psline{C-C}(1,0)}} \psline(0,0)(0,-7)
    \cell(1,1)[0] \cell(2,1)[1] \cell(2,2)[1] \cell(3,1)[0] \cell(3,2)[0]
    \cell(3,3)[0] \cell(4,1)[0] \cell(4,2)[1] \cell(4,3)[0] \cell(4,4)[1]
    \cell(5,1)[1] \cell(5,2)[1] \cell(5,3)[1] \cell(6,1)[0] \cell(6,2)[1]
    \tiny
    \rput(0,-4){
      \tiny    \rput(3.5,-.2){1}
      \rput(2.5,-1.2){3}
      \rput(1.5,-2.2){5}
      \rput(0.5,-2.2){6}
      \rput(0.2,-2.5){7}
      \rput(2.2,-1.5){4}
      \rput(3.2,-0.5){2}
    }
  \psline[linecolor=red,arrowsize=.4,
  arrowlength=.6]{->}(1.5,-6)(1.5,-1.5)(0.5,-1.5)(0.5,-4.5)(1.3,-4.5)(1.3,-5.5)(2,-5.5)
\end{pspicture}
\hspace{2cm}
  \begin{pspicture}(0,0)(4,-7)
    \multido{\n=0+1}{4}{\rput(\n,-\n){\psline{C-C}(1,0)}} \psline(0,0)(0,-7)
    \cell(1,1)[0] \cell(2,1)[1] \cell(2,2)[1] \cell(3,1)[0] \cell(3,2)[0]
    \cell(3,3)[0] \cell(4,1)[0] \cell(4,2)[1] \cell(4,3)[0] \cell(4,4)[1]
    \cell(5,1)[1] \cell(5,2)[1] \cell(5,3)[1] \cell(6,1)[0] \cell(6,2)[1]
    \tiny
    \rput(0,-4){
      \tiny    \rput(3.5,-.2){1}
      \rput(2.5,-1.2){3}
      \rput(1.5,-2.2){5}
      \rput(0.5,-2.2){6}
      \rput(0.2,-2.5){7}
      \rput(2.2,-1.5){4}
      \rput(3.2,-0.5){2}
    }
  \psline[linecolor=red,arrowsize=.4,
  arrowlength=.6]{->}(0.5,-6)(0.5,-1.5)(1.5,-1.5)(1.5,-3.5)(3.5,-3.5)(3.5,-4)
\end{pspicture}
\caption{An illustration of the zigzag bijection $\Phi_B$.}
\label{fig:ptb}
\end{figure}

Let $\pi=\pi_1\cdots\pi_n\in B_n$. An integer $i\in[n]$ is called a \emph{weak
  excedance} if $\pi_i\geq i$, and an \emph{excedance} if $\pi_i > i$. An
integer $i\in[n-1]$ is called a \emph{descent} of $\pi$ if
$\pi_i>\pi_{i+1}$. There are various statistics on type $B$ permutations
\cite{adin,bjorner,foata}:
\begin{align*}
\wex(\pi) &= \#\{i\in[n]: \pi_i\geq i\},\\
\exc(\pi) &= \#\{i\in[n]: \pi_i > i\},\\
\des(\pi) &= \#\{i\in[n-1] : \pi_i >\pi_{i+1}\},\\
\des_B(\pi) &= \#\{i\in[0,n-1] : \pi_i >\pi_{i+1}\},\quad 
\mbox{where $\pi_0=0$},\\      
\neg(\pi) &= \#\{i\in[n]: \pi_i<0\},\\
\fwex(\pi) &= 2\wex(\pi)+\neg(\pi),\\
\fexc(\pi) &= 2\exc(\pi)+\neg(\pi),\\
\fdes(\pi) &= \des(\pi) + \des_B(\pi).
\end{align*}

There is another statistic $\cro(\pi)$ that was first defined in \cite{CJW}.
\begin{defn} \cite{CJW} A \emph{crossing} of a signed permutation
  $\pi=\pi_1\cdots\pi_n$ is a pair $(i,j)$ with $i,j>0$ such that
\begin{itemize}
\item $i<j\le \pi_i<\pi_j$ or
\item $-i<j\le -\pi_i<\pi_j$ or
\item $i>j>\pi_i>\pi_j$.
\end{itemize}
We denote by $\cro(\pi)$ the number of crossings of $\pi$. 
\end{defn}

Steingr\'imsson and Williams \cite{Steingrimsson2007} showed that if $\pi =
\Phi(T)$, then $\so(T) = \cro(\pi)$ and $\wex(\pi)=\row(T)$. Corteel et
al.~\cite{CJW} found the following type $B$ analog of this result. 

 \begin{prop}\label{prop:zigzag}\cite[Theorem~4]{CJW}
   The zigzag map $\Phi_B:\PT_B(n)\to B_n$ is a bijection. Moreover, if $\pi =
   \Phi_B(T)$, then $\wex(\pi)=\row(T)$, $\neg(\pi)=\diag(T)$, and
   $\cro(\pi)=\so(T)$.
 \end{prop}

 \begin{remark}
 In \cite[Theorem~4]{CJW} Corteel et al.~only show
 $\neg(\pi)=\diag(T)$ and $\cro(\pi)=\so(T)$. However, one can easily see from
 the definition of $\Phi_B$ that we also have $\wex(\pi)=\row(T)$.
 \end{remark}

\subsection{Equidistribution of some statistics in $B_n$}

We now find a relation between $\fwex$, $\fexc$, and $\fdes$. For permutations,
it is well known \cite[1.4.3~Proposition]{EC1} that 
\begin{equation}
  \label{eq:10}
\#\{\pi\in S_n: \wex(\pi)=k\} = \#\{\pi\in S_n: \des(\pi)=k-1\}
= \#\{\pi\in S_n: \exc(\pi)=k-1\}.
\end{equation}

We will use the following result of Foata and Han.
\begin{lem}\cite[Section~9]{foata}\label{lem:FH}
There is a bijection $\psi:B_n\to B_n$ such that $\fexc(\pi) =
\fdes(\psi(\pi))$.   
\end{lem}

For $\pi=\pi_1\cdots\pi_n\in B_n$, we define $-\pi\in B_n$ by
$(-\pi)_i=-(\pi_i)$. We also define $\pi^{\tr}\in B_n$ to be the signed
permutation such that $\pi^{\tr}_i =\epsilon \cdot j$ if and only $\pi_j =
\epsilon \cdot i$ for $\epsilon\in\{1,-1\}$ and $i,j\in [n]$. In other words, if
$M(\pi)$ is the signed permutation matrix of $\pi$, then $M(-\pi)=-M(\pi)$ and
$M(\pi^{\tr}) = M(\pi)^{\tr}$. Here, the \emph{signed permutation matrix}
$M(\pi)$ is the $n\times n$ matrix whose $(i,j)$-entry is $1$ if $\pi_i=j$, $-1$
if $\pi_i=-j$, and $0$ otherwise. The following lemma is easy to prove.

 \begin{lem}\label{lem:fwex}
For $\pi\in B_n$, we have
\[
\fdes(\pi) + \fdes(-\pi) = 2n-1, \qquad 
\fwex(\pi) + \fexc(\pi^{\tr}) = 2n.
\]
 \end{lem}

 We now are ready to prove a type $B$ analog of \eqref{eq:10}.
\begin{prop}\label{prop:fwex_desB}
We have
\begin{equation}
  \label{eq:8}
\#\{\pi\in B_n: \fwex(\pi)=k\} = \#\{\pi\in B_n: \fdes(\pi)=k-1\}
= \#\{\pi\in B_n: \fexc(\pi)=k-1\},
\end{equation}
and
\begin{equation}
  \label{eq:9}
\#\{\pi\in B_n: \des_B(\pi)=k\} = \#\{\pi\in B_n: \floor{\fwex(\pi)/2}=k\}.  
\end{equation}
\end{prop}
\begin{proof}
By Lemmas~\ref{lem:FH} and \ref{lem:fwex}, we have  
\begin{align*}
\#\{\pi\in B_n: \fwex(\pi)=k\} &= \#\{\pi\in B_n: \fexc(\pi)=2n-k\}
= \#\{\pi\in B_n: \fdes(\pi)=2n-k\}\\
&= \#\{\pi\in B_n: \fdes(\pi)=2n-1-(2n-k)\}\\
&= \#\{\pi\in B_n: \fdes(\pi)=k-1\}=\#\{\pi\in B_n: \fexc(\pi)=k-1\}.
\end{align*}
Equation~\eqref{eq:9} follows from Equation~\eqref{eq:8} and the fact that
$\des_B(\pi) = \flr{(\fdes(\pi)+1)/2}$.
\end{proof}

\section{\texorpdfstring{$q$-Eulerian numbers of type $B$}{q-Eulerian numbers of type B}}
\label{sec:qEuler}

Let $k$ and $n$ be integers with $1\leq k \leq n $.  The (type $A$)
\emph{Eulerian number} $E_{n,k}$ is the number of $\pi\in S_n$ with
$\des(\pi)=k-1$. Equivalently, $E_{n,k}$ is the number of $\pi\in S_n$ with
$\wex(\pi)=k$.  The \emph{$q$-Eulerian number} $E_{n,k}(q)$ originally defined
in \cite{Williams2005} is as follows:
\[ 
  E_{n,k}(q) = B_{n,2k}(0,q) = \sum_{\substack{\pi\in S_n\\ \wex(\pi)=k}} q^{\cro(\pi)}.
\]

Now let $k$ and $n$ be integers with $0\leq k \leq n$.  The \emph{type~$B$
  Eulerian number} $E_{n,k}^B$ is the number of $\pi\in B_n$ with
$\des_B(\pi)=k$. By Equation~\eqref{eq:9}, $E_{n,k}^B$ is also the
number of $\pi\in B_n$ with $\flr{\fwex(\pi)/2}=k$.

\begin{defn}
  For $0\leq k \leq n$, we define the \emph{type~$B$ $q$-Eulerian number}
  $E_{n,k}^B(q)$ as follows:
\[ E_{n,k}^B(q) =\sum_{\substack{\pi\in B_n\\ \flr{\fwex(\pi)/2}=k}} q^{\cro(\pi)}.\]
\end{defn}

By Theorem~\ref{thm:ZZ} we can write
\begin{equation}
  \label{eq:7}
E_{n,k}^B(q) = B_{n,2k}(1,q) + B_{n,2k+1}(1,q).  
\end{equation}
By Proposition~\ref{prop:zigzag} we have a different expression for
$E_{n,k}^B(q)$ using permutation tableaux of type $B$:
\[ 
E_{n,k}^B(q) =\sum_{\substack{T\in\PT_B(n)\\ \row(T)+\floor{\diag(T)/2}=k}} q^{\so(T)}.
\]

The $q$-Eulerian number $E_{n,k}(q)$ becomes the binomial coefficient $\binom{n-1}{k-1}$
when $q=-1$, the Narayana number $\frac{1}{n}
\binom{n}{k} \binom{n}{k-1}$ when $q=0$, and the Eulerian number $E_{n,k}$ when
$q=1$. Thus we have $E_{n,k}(q) = E_{n,n+1-k}(q)$ in the special cases $q\in\{-1,0,1\}$. 
More generally:
\begin{prop}[Williams \cite{Williams2005}]\label{thm:symA}
For any $1\leq k\leq n$, we have
  \[ E_{n,k}(q) = E_{n,n+1-k}(q).\]
\end{prop}

Our $q$-analog of the type $B$ Eulerian numbers also has such properties.

\begin{thm}[Theorem~\ref{qeulrianb} in the introduction]
 For any $0\leq k \leq n$, we have $E_{n,k}^B(-1)=\binom nk$,  $E_{n,k}^B(0)=\binom nk^2$ (the Narayana numbers of type $B$),
$E_{n,k}(1)$ is the type $B$ Eulerian number, and 
\begin{equation}
  \label{eq:11}
 E^B_{n,k}(q)=E^B_{n,n-k}(q).
\end{equation}
\end{thm}


The values at $q=-1$ and $q=0$ will be obtained respectively in Section~\ref{MA}
and Section~\ref{EF}.  Recall that
\[
B_{n,k}^*(t,q) =\sum_{\substack{\pi\in B_n\\ \fwex(\pi)=k}} 
t^{\neg(\pi) + \chi(\pi_1>0)} q^{\cro(\pi)}.
\]
Since $B_{n,k}^*(1,q)=B_{n,k}(1,q)$, the symmetric property \eqref{eq:11} is a
direct consequence of \eqref{eq:7} together with the following result:

\begin{thm}[Theorem \ref{bnsym} in the introduction]  \label{thm:sym}
For $1\leq k\leq 2n$, we have 
\[ B_{n,k}^*(t,q)=B_{n,2n+1-k}^*(t,q).\]  
\end{thm}


In order to prove Theorem~\ref{thm:sym} we introduce a diagram representing a
signed permutation.

\subsection{Pignose diagrams}
Given a set $U$ of $2n$ distinct integers, an \emph{ordered matching} on $U$ is
a set of ordered pairs $(i,j)$ of integers such that each integer in $U$ appears
exactly once. For an ordered matching $M$ on $U$ containing $2n$ integers
$a_1<a_2<\cdots<a_{2n}$, we define the \emph{standardization} $\st(M)$ of $M$ to
be the ordered matching on $[2n]$ obtained from $M$ by replacing $a_i$ with $i$
for each $i\in[2n]$.  For example, if $M=\{(2,6),(5,3),(9,4), (7,8)\}$,
then $\st(M)$ is the ordered matching $\{(1,5),(4,2),(8,3), (6,7)\}$ on $[8]$.

We represent an ordered matching $M$ on $U$ as follows. Arrange the integers in
$U$ on a horizontal line in increasing order. For each pair $(i,j)\in M$,
connect $i$ and $j$ with an upper arc if $i<j$, and with a lower arc if
$i>j$. For example, the following represents $\{(1,5),(4,2),(8,3), (6,7)\}$.
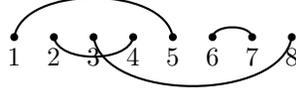
\[
\begin{pspicture}(0,-.5)(9,2)
\multido{\n=1+1}{8}{\vput{\n} \rput(\n,.5){$\n$}}
\edge15 \edge67 \edge 42 \edge 83
\end{pspicture}
\]
We also define a \emph{crossing} of an ordered matching $M$ to be two
intersecting arcs and denote by $\cro(M)$ the number of crossings of $M$.

For $\pi=\pi_1\cdots\pi_n\in B_n$, the \emph{pignose diagram} of $\pi$ is
defined as follows. First we arrange $2n$ vertices in a horizontal line where
the $(2i-1)$th vertex and the $2i$th vertex are enclosed by an ellipse labeled
with $i$ which we call the $i$th \emph{pignose}. The left vertex and the right
vertex in a pignose are called the \emph{first vertex} and the \emph{second
  vertex} respectively.  For each $i\in[n]$, we connect the first vertex of the
$i$th pignose and the second vertex of the $\pi_i$th pignose with an arc in the
following way. If $\pi_i>0$, then draw an arc above the horizontal line if
$\pi_i\ge i$ and below the horizontal line if $\pi_i<i$.  If $\pi_i<0$, then we
draw an arc starting from the first vertex of the $i$th pignose below the
horizontal line to the second vertex of the $\pi_i$th pignose above the
horizontal line like a spiral oriented clockwise. We draw these spiral arcs so
that these are not crossing each other below the horizontal line. For example,
the following is the pignose diagram of $\pi=(4, -6, 1, -5, -3, 7, 2)$.
\[
\psset{unit=15pt}
\begin{pspicture}(-4,-2)(13,4) 
\vvput{-0.3}{1a} \vvput{0.3}{1b} \rput(0,.5){$1$} \pig{0} \vvput{1.7}{2a} \vvput{2.3}{2b} \rput(2,.5){$2$} \pig{2} \vvput{3.7}{3a} \vvput{4.3}{3b} \rput(4,.5){$3$} \pig{4} \vvput{5.7}{4a} \vvput{6.3}{4b} \rput(6,.5){$4$} \pig{6} \vvput{7.7}{5a} \vvput{8.3}{5b} \rput(8,.5){$5$} \pig{8} \vvput{9.7}{6a} \vvput{10.3}{6b} \rput(10,.5){$6$} \pig{10} \vvput{11.7}{7a} \vvput{12.3}{7b} \rput(12,.5){$7$} \pig{12} \pnode(-1.5,1){m1a} \pnode(-1.5,1){m1b} \pnode(-2.5,1){m2a} \pnode(-2.5,1){m2b} \pnode(-3.5,1){m3a} \pnode(-3.5,1){m3b} \edge{1a}{4b} \edge{2a}{m1b} \edge{m1a}{6b} \edge{3a}{1b} \edge{4a}{m2b} \edge{m2a}{5b} \edge{5a}{m3b} \edge{m3a}{3b} \edge{6a}{7b} \edge{7a}{2b} 
\end{pspicture}
\]

For $\pi\in B_n$, one can easily check that
\begin{itemize}
\item $\cro(\pi)$ is  the number of unordered
 pairs of two arcs crossing each other in the pignose diagram of $\pi$, 
\item $\fwex(\pi)$ is twice the number of upper arcs plus the number of spiral
  arcs, or equivalently, the number of vertices with a half arc above the
  horizontal line, 
\item $\neg(\pi)$ is the number of spiral arcs. 
\end{itemize}

Since $S_n$ is contained in $B_n$, the pignose diagram for $\pi\in S_n$ is also
defined. Note that the pignose diagram of $\pi\in S_n$ can be considered as an
ordered matching on $[2n]$ by removing the ellipses enclosing two vertices and
labeling the $2n$ vertices with $1,2\ldots,2n$ from left to right.  We call an
ordered matching that can be obtained in this way a \emph{pignose matching}.
Not all ordered matchings are pignose matchings.  One can easily prove the
following proposition.

\begin{prop}\label{thm:OM}
  An ordered matching $M$ on $[2n]$ is a pignose matching if and only if in the
  representation of $M$ each odd integer has a half arc of type
  $\halfarc(.5,.5)$ or $\halfarc(-.5,-.5)$, and each even integer has a half arc
  of type $\halfarc(-.5,.5)$ or $\halfarc(.5,-.5)$.
\end{prop}

The following lemma was first observed by M\'edicis and Viennot \cite[Lemme~3.1
Cas b.4]{MV}. For reader's convenience we include a proof as well. 

\begin{lem}\label{lem:divide}
  For $\pi\in S_n$ and an integer $k\in[n]$, the number of integers $i\in[n]$
  with $i\leq k\leq \pi_i$ is equal to the number of integers $i\in[n]$ with
  $\pi_i < k < i$ plus $1$.  Pictorially, this means that in the pignose diagram
  of $\pi$ if we draw a vertical line dividing the $k$th pignose in the middle,
  then the number of upper arcs intersecting with this line is equal to the
  number of lower arcs intersecting with this line plus $1$ as shown below.
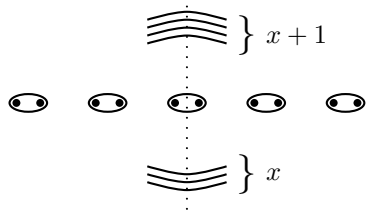
\[
\begin{pspicture}(-1,-2)(9,4) 
\vvput{-0.3}{1a} \vvput{0.3}{1b} \pig{0}
  \vvput{1.7}{2a} \vvput{2.3}{2b} \pig{2} \vvput{3.7}{3a} \vvput{4.3}{3b}
  \pig{4} \vvput{5.7}{4a} \vvput{6.3}{4b} \pig{6} \vvput{7.7}{5a}
  \vvput{8.3}{5b}\pig{8} 
\psline[linestyle=dotted](4,-1.7)(4,3.7)
\rput(4,2.5){\multido{\n=0+2}{4}{\rput(0,0.\n){\pscurve(-1,0)(0,.2)(1,0)}}}
\rput(4,-1){\multido{\n=0+2}{3}{\rput(0,0.\n){\pscurve(-1,0)(0,-.2)(1,0)}}}
\rput(5.5,2.7){\psscaleboxto(.5,1){\}}} \rput[l](6,2.7){$x+1$}
\rput(5.5,-.8){\psscaleboxto(.5,1){\}}} \rput[l](6,-.8){$x$}
\end{pspicture}
\]
\end{lem}
\begin{proof}
  Consider the pignose diagram of $\pi$ with a vertical line dividing the
  pignose of $k$ in the middle. Let $x$ (resp.~$y$) be the number of lower
  (resp.~upper) arcs intersecting with the vertical line. Let $a, b,c,d$ be,
  respectively, the number of half arcs to the left of the vertical line of the
  following types:
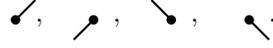
\[\halfarc(.5,.5), \quad
\halfarc(-.5,-.5), \quad
\halfarc(-.5,.5),\quad
\halfarc(.5,-.5).\]

For each pignose, the first vertex has a half edge of the first or the second
type and the second vertex has a half edge of the third or the fourth type. Thus
$a+b = c+d+1$. Since $x=a-c$ and $y=d-b$, we are done.
\end{proof}

\subsection{Proof of Theorem~\ref{thm:sym}}
 
Let
\[
  B_n^+  = \{\pi\in B_n: \pi_1>0\},\qquad
  B_n^-  = \{\pi\in B_n: \pi_1<0\},
\]
\[
B_{n,k}^+(t,q)=\sum_{\substack{\pi\in B_n^+\\ \fwex(\pi)=k}} 
t^{\neg(\pi)}q^{\cro(\pi)}, \qquad
B_{n,k}^-(t,q)=\sum_{\substack{\pi\in B_n^-\\ \fwex(\pi)=k}} 
t^{\neg(\pi)}q^{\cro(\pi)}.
\]
Then 
\begin{equation}
  \label{eq:13}
B_{n,k}^*(t,q) = t B_{n,k}^+(t,q) + B_{n,k}^-(t,q).  
\end{equation}
In this subsection we will prove the following proposition.

\begin{prop}\label{prop:bij}
  There is a bijection $\phi:B_n^+\to B_n^-$ such that 
\begin{equation}
  \label{eq:2}
\cro(\phi(\pi))=\cro(\pi),\qquad \neg(\phi(\pi))=\neg(\pi)+1, 
\qquad \fwex(\phi(\pi))=2n+1-\fwex(\pi).
\end{equation}
Thus,
\begin{equation}
  \label{eq:14}
t B_{n,k}^+(t,q) = B_{n,2n+1-k}^-(t,q), \qquad
B_{n,k}^-(t,q) = t B_{n,2n+1-k}^+(t,q). 
\end{equation}
\end{prop}
Note that Theorem~\ref{thm:sym} follows from \eqref{eq:13} and \eqref{eq:14}.
In order to prove Proposition~\ref{prop:bij} we need some definitions and
lemmas.

Given an ordered matching $M$ on $[2n]$, we define $\rho(M)$ to be the ordered
matching $\st(M')$ where $M'$ is the ordered matching on $\{2,3,\dots,2n+1\}$
obtained from $M$ by replacing $1$ with $2n+1$.  For example, if
$M=\{(1,5),(4,2),(8,3), (6,7)\}$, then $M'=\{(9,5),(4,2),(8,3), (6,7)\}$ and
$\rho(M)=\st(M')=\{(8,4),(3,1),(7,2), (5,6)\}$.  Pictorially, $\rho(M)$ is
obtained from $M$ by moving the first vertex to the end and reflecting the arc
adjacent to this vertex as follows.
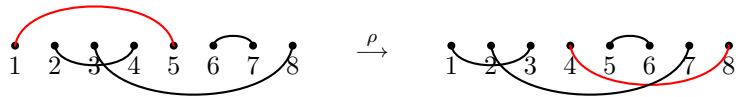
\[
\begin{pspicture}(0,-.5)(10,2)
\multido{\n=1+1}{8}{\vput{\n} \rput(\n,.5){$\n$}}
\edge67 \edge 42 \edge 83 
\psset{linecolor=red}
\edge15
\rput(10,1){$\displaystyle \substack{\rho\\ \longrightarrow}$}
\end{pspicture} 
\begin{pspicture}(-1,-.5)(9,2)
\multido{\n=1+1}{8}{\vput{\n} \rput(\n,.5){$\n$}}
{\psset{linecolor=red}
\edge84} \edge31 \edge72 \edge56
\end{pspicture}
\]

We denote $\rho^{(k)}=\overbrace{\rho\circ\cdots\circ\rho}^k$.

\begin{lem}\label{thm:moving}
  Let $M$ be a pignose matching on $[2n]$. Then $\rho^{(k)}(M)$ has the same
  number of crossings as $M$ for all positive integers $k$.
\end{lem}
\begin{proof}
  Note that $\rho^{(2)}(M)$ is a pignose matching. Thus it suffices to prove for
  $k=1$ and $k=2$.

  Considering $M$ as a pignose diagram of a permutation in $S_n$, assume that
  $1$ is connected to $i$. By Lemma~\ref{lem:divide}, if we draw a vertical line
  between the two vertices in the $i$th pignose, the number, say $x$, of upper
  arcs above $i$ except $(1,i)$ is equal to the number of lower arcs below
  $i$. Therefore, when we go from $M$ to $\rho(M)$, we lose $x$ crossings and
  obtain new $x$ crossings as shown below. 
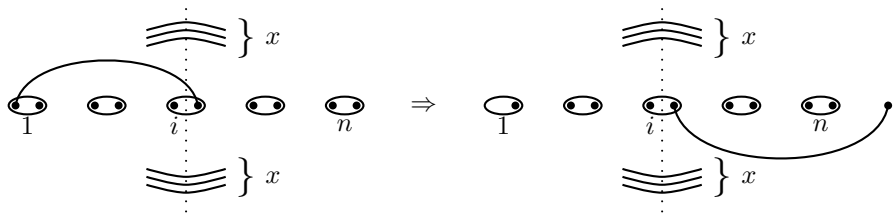
\[
\begin{pspicture}(-1,-2)(10,4) 
\vvput{-0.3}{1a} \vvput{0.3}{1b} \pig{0}
  \vvput{1.7}{2a} \vvput{2.3}{2b} \pig{2} \vvput{3.7}{3a} \vvput{4.3}{3b}
  \pig{4} \vvput{5.7}{4a} \vvput{6.3}{4b} \pig{6} \vvput{7.7}{5a}
  \vvput{8.3}{5b}\pig{8} 
\psline[linestyle=dotted](4,-1.7)(4,3.7)
\rput(4,2.5){\multido{\n=0+2}{3}{\rput(0,0.\n){\pscurve(-1,0)(0,.2)(1,0)}}}
\rput(4,-1){\multido{\n=0+2}{3}{\rput(0,0.\n){\pscurve(-1,0)(0,-.2)(1,0)}}}
\rput(5.5,2.7){\psscaleboxto(.5,1){\}}} \rput[l](6,2.7){$x$}
\rput(5.5,-.8){\psscaleboxto(.5,1){\}}} \rput[l](6,-.8){$x$}
\rput(0,.5){$1$} \rput(3.7,.5){$i$} \rput(8,.5){$n$}
\edge{1a}{3b}
\rput(10,1){$\Rightarrow$}
\end{pspicture}
\begin{pspicture}(-2,-2)(9,4) 
\vvput{0.3}{1b} \pig{0}
  \vvput{1.7}{2a} \vvput{2.3}{2b} \pig{2} \vvput{3.7}{3a} \vvput{4.3}{3b}
  \pig{4} \vvput{5.7}{4a} \vvput{6.3}{4b} \pig{6} \vvput{7.7}{5a}
  \vvput{8.3}{5b}\pig{8} 
\psline[linestyle=dotted](4,-1.7)(4,3.7)
\rput(4,2.5){\multido{\n=0+2}{3}{\rput(0,0.\n){\pscurve(-1,0)(0,.2)(1,0)}}}
\rput(4,-1){\multido{\n=0+2}{3}{\rput(0,0.\n){\pscurve(-1,0)(0,-.2)(1,0)}}}
\rput(5.5,2.7){\psscaleboxto(.5,1){\}}} \rput[l](6,2.7){$x$}
\rput(5.5,-.8){\psscaleboxto(.5,1){\}}} \rput[l](6,-.8){$x$}
\rput(0,.5){$1$} \rput(3.7,.5){$i$} \rput(8,.5){$n$}
\vvput{9.7}{6a}
\edge{6a}{3b}
\end{pspicture}
\]
This proves the assertion for $k=1$.

To prove for $k=2$ we define the following. Given an ordered matching $N$, let
$\overline N$ be the ordered matching obtained by reflecting $N$ along the
horizontal line. It is easy to see that $\cro(N)=\cro(\overline N)$ and
$\overline{\rho(N)}=\rho(\overline N)$. Moreover if $N$ is a pignose matching
then so is $\overline{\rho(N)}$. Thus we have
  \begin{equation}
    \label{eq:1}
 \cro\left(\rho^{(2)}(M)\right)
=\cro\left(\overline{\rho^{(2)}(M)}\right)
=\cro\left(\rho\left(\overline{\rho(M)}\right)\right)
 \end{equation}

 Since $\overline{\rho(M)}$ is a pignose matching, using the assertion for
 $k=1$, we obtain that \eqref{eq:1} is equal to
\[ 
\cro\left( \overline{\rho(M)}\right)
=\cro\left( \rho(M)\right)
=\cro\left( M\right).
\]
Thus $\cro\left(\rho^{(2)}(M)\right)=\cro\left( M\right)$ and we are done.
\end{proof}

For an ordered matching $M$ on $[2n]$ we define $M^\rev$ to be the ordered
matching obtained from $M$ by replacing $i$ with $2n+1-i$ for each
$i\in[2n]$. Pictorially $M^\rev$ is obtained from $M$ by taking a $180^\circ$
rotation.

\begin{lem}\label{thm:transform}
  Let $M$ be a pignose matching on $[2n]$. Then $N=(\rho^{(2k+1)}(M))^\rev$ is
  also a pignose matching and $\cro(M)=\cro(N)$. 
\end{lem}
\begin{proof}
  Note that $\rho^{(2)}(M)$ is a pignose matching. Thus it suffices to prove
  that $(\rho(M))^\rev$ is a pignose matching, which easily follows from
  Proposition~\ref{thm:OM}.
\end{proof}



For $\pi\in B_n$, let $\pi^-=(-\pi_1)\pi_2\cdots \pi_n$.

\begin{lem}\label{thm:first_element}
  The map $\pi\mapsto\pi^-$ is a bijection from $B_n^+$ to $B_n^-$.  Moreover,
  we have $\cro(\pi)=\cro(\pi^-)$, $\neg(\pi)= \neg(\pi^-)-1$, and $\fwex(\pi)=
  \fwex(\pi^-)+1$.
\end{lem}
\begin{proof}
  This is an immediate consequence of the following observation: if $\pi\in
  B_n^+$, the pignose diagram of $\pi^-$ is obtained from the pignose diagram of
  $\pi$ by changing the upper arc adjacent to $1$ to a spiral arc as follows.
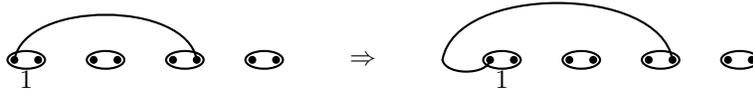
\[
\begin{pspicture}(-1,0)(7,2.5) 
 \vvput{-0.3}{1a} \vvput{0.3}{1b} \rput(0,.5){$1$} \pig{0} \vvput{1.7}{2a}
  \vvput{2.3}{2b} \pig{2} \vvput{3.7}{3a} \vvput{4.3}{3b}
\pig{4} \vvput{5.7}{4a} \vvput{6.3}{4b} \pig{6}
\edge{1a}{3b} 
\end{pspicture}
\begin{pspicture}(-5,0)(7,2.5)
\rput(-3.5,1){$\Rightarrow$}
  \vvput{-0.3}{1a} \vvput{0.3}{1b} \rput(0,.5){$1$} \pig{0} \vvput{1.7}{2a}
  \vvput{2.3}{2b} \pig{2} \vvput{3.7}{3a} \vvput{4.3}{3b}
  \pig{4} \vvput{5.7}{4a} \vvput{6.3}{4b}\pig{6} 
  \pnode(-1.5,1){m1a}
  \pnode(-1.5,1){m1b} \edge{m1b}{3b}  \edge{1a}{m1a}
\end{pspicture}
\]
\end{proof}

Now we are ready to define the map $\phi:B_n^+\to B_n^-$ in
Proposition~\ref{prop:bij}.  Suppose $\pi\in B_n^+$ and $\neg(\pi)=m$. We make
the pignose diagram of $\pi$ to be a pignose matching on $[2m+2n]$ by dividing
each spiral arc into one upper arc and one lower arc so that the left endpoint
of the upper arc is to the left of the left endpoint of the lower arc as shown
below.
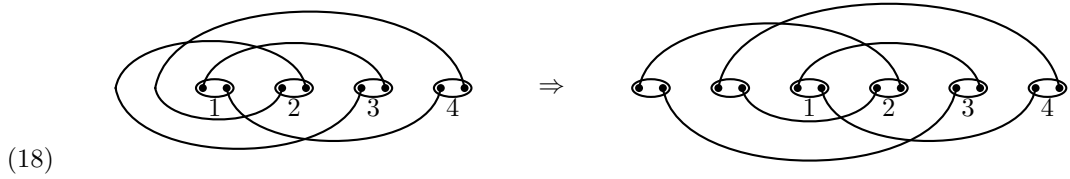
\begin{equation}
  \label{eq:3}
\begin{pspicture}(-3,-1)(7,3.3)
\vvput{-0.3}{1a} \vvput{0.3}{1b} \rput(0,.5){$1$} \pig{0} \vvput{1.7}{2a} \vvput{2.3}{2b} \rput(2,.5){$2$} \pig{2} \vvput{3.7}{3a} \vvput{4.3}{3b} \rput(4,.5){$3$} \pig{4} \vvput{5.7}{4a} \vvput{6.3}{4b} \rput(6,.5){$4$} \pig{6} \pnode(-1.5,1){m1a} \pnode(-1.5,1){m1b} \pnode(-2.5,1){m2a} \pnode(-2.5,1){m2b} \edge{1a}{3b} \edge{2a}{m1b} \edge{m1a}{4b} \edge{3a}{m2b} \edge{m2a}{2b} \edge{4a}{1b} 
\end{pspicture}
\begin{pspicture}(-8,-1)(7,3.3) 
\rput(-6.5,1){$\Rightarrow$}
\vvput{-0.3}{1a}
  \vvput{0.3}{1b} \rput(0,.5){$1$} \pig{0} \vvput{1.7}{2a} \vvput{2.3}{2b}
  \rput(2,.5){$2$} \pig{2} \vvput{3.7}{3a} \vvput{4.3}{3b} \rput(4,.5){$3$}
  \pig{4} \vvput{5.7}{4a} \vvput{6.3}{4b} \rput(6,.5){$4$} \pig{6}
  \vvput{-2.3}{m1a} \vvput{-1.7}{m1b} 
  \vvput{-4.3}{m2a} \vvput{-3.7}{m2b} \edge{1a}{3b}
  \edge{2a}{m1b} \edge{m1a}{4b} \edge{3a}{m2b} \edge{m2a}{2b} \edge{4a}{1b}
\pig{-2} \pig{-4}
\end{pspicture}
\end{equation}

Let $M$ be the pignose matching obtained in this way, and let
$N=\left(\rho^{(2m+1)}(M)\right)^\rev$.  By Lemma~\ref{thm:transform}, $N$ is
also a pignose matching on $[2m+2n]$ and $\cro(M)=\cro(N)$. It is
straightforward to check that $N$ satisfies the following properties.
\begin{enumerate}
\item For each $i\in[2m]$, the $i$th vertex is connected to the $j$th vertex for
  some $j>2m$.
\item The first $m$ lower arcs do not cross each other.
\item The $(2m+1)$st vertex has an upper half arcs.
\item The number of upper half arcs adjacent to the last $2n$ vertices is
  $2n+2-k$. 
\end{enumerate}
By the first and the second properties, we can make $N$ to be the pignose
diagram of a signed permutation, say $\sigma\in B_n$, by identifying the
$(2i-1)$th vertex and the $(2i)$th vertex for each $i\in[m]$. Then
$\neg(\pi)=\neg(\sigma)$. By the third property, we have $\sigma\in B_n^+$, and
by the fourth property, we have $\fwex(\sigma)=2n+2-k$. We define $\phi(\pi)$ to
be $\sigma^-$. Clearly $\phi$ is a bijection from $B_n^+$ to $B_n^-$. By
Lemma~\ref{thm:first_element}, $\phi$ satisfies \eqref{eq:2}. This finishes the
proof of Proposition~\ref{prop:bij}.

\begin{example}
  Let $\pi=3,-4,-2,1\in B_4^+$. The pignose diagram of $\pi$ and $M$ are shown
  in \eqref{eq:3}. Then we can compute $\rho^{(5)}(M)$,
  $\left(\rho^{(5)}(M)\right)^\rev$, $\sigma$, and $\sigma^-$ as follows.
 
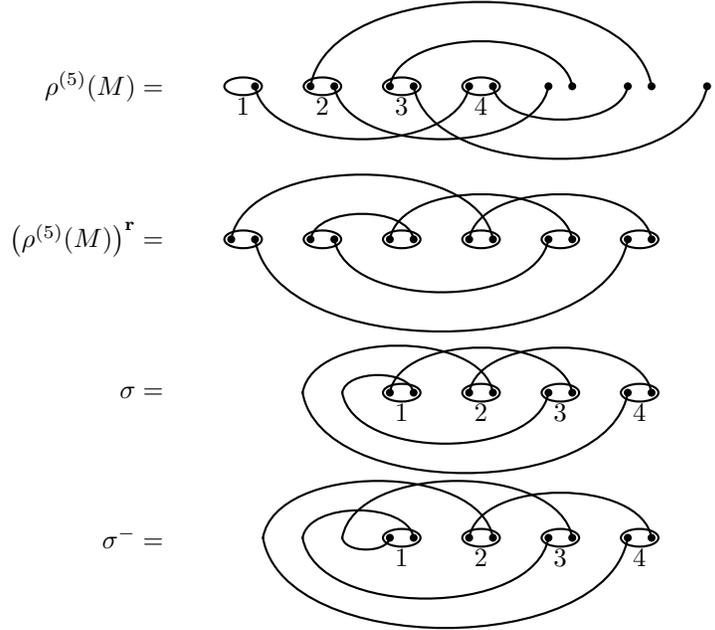
\begin{center}
\begin{pspicture}(-6,-1)(12,3.3) 
\rput[r](-2,1){$\rho^{(5)}(M)=$}
 \vvput{0.3}{1b} \rput(0,.5){$1$} \pig{0} \vvput{1.7}{2a} \vvput{2.3}{2b}
  \rput(2,.5){$2$} \pig{2} \vvput{3.7}{3a} \vvput{4.3}{3b} \rput(4,.5){$3$}
  \pig{4} \vvput{5.7}{4a} \vvput{6.3}{4b} \rput(6,.5){$4$} \pig{6}
 \vvput{7.7}{5a}  \vvput{9.7}{6a}  \vvput{11.7}{7a}
 \vvput{8.3}{5b}  \vvput{10.3}{6b} 
\edge{7a}{3b}
  \edge{2a}{6b} \edge{6a}{4b} \edge{3a}{5b} \edge{5a}{2b} \edge{4a}{1b}
\end{pspicture}

\begin{pspicture}(-6,-1.5)(12,2.8) 
  \rput[r](-2,1){$\left(\rho^{(5)}(M)\right)^\rev=$} \vvput{-0.3}{1a}
  \vvput{0.3}{1b} \pig{0} \vvput{1.7}{2a} \vvput{2.3}{2b}
\pig{2} \vvput{3.7}{3a} \vvput{4.3}{3b} 
  \pig{4} \vvput{5.7}{4a} \vvput{6.3}{4b} \pig{6}
  \vvput{7.7}{5a} \vvput{8.3}{5b} \pig{8} \vvput{9.7}{6a}
  \vvput{10.3}{6b} \pig{10} \edge{1a}{4b} \edge{2a}{3b}
  \edge{3a}{5b} \edge{4a}{6b} \edge{5a}{2b} \edge{6a}{1b}
\end{pspicture}

\begin{pspicture}(-6,-1)(12,2.3) 
  \rput[r](-2,1){$\sigma=$} 
\rput(4,0){
\vvput{-0.3}{1a} \vvput{0.3}{1b} \rput(0,.5){$1$} \pig{0} \vvput{1.7}{2a} \vvput{2.3}{2b} \rput(2,.5){$2$} \pig{2} \vvput{3.7}{3a} \vvput{4.3}{3b} \rput(4,.5){$3$} \pig{4} \vvput{5.7}{4a} \vvput{6.3}{4b} \rput(6,.5){$4$} \pig{6} \pnode(-1.5,1){m1a} \pnode(-1.5,1){m1b} \pnode(-2.5,1){m2a} \pnode(-2.5,1){m2b} \edge{1a}{3b} \edge{2a}{4b} \edge{3a}{m1b} \edge{m1a}{1b} \edge{4a}{m2b} \edge{m2a}{2b} }
\end{pspicture}

\begin{pspicture}(-6,-1.5)(12,2.6) 
  \rput[r](-2,1){$\sigma^-=$} 
\rput(4,0){
\vvput{-0.3}{1a} \vvput{0.3}{1b} \rput(0,.5){$1$} \pig{0} \vvput{1.7}{2a} \vvput{2.3}{2b} \rput(2,.5){$2$} \pig{2} \vvput{3.7}{3a} \vvput{4.3}{3b} \rput(4,.5){$3$} \pig{4} \vvput{5.7}{4a} \vvput{6.3}{4b} \rput(6,.5){$4$} \pig{6} \pnode(-1.5,1){m1a} \pnode(-1.5,1){m1b} \pnode(-2.5,1){m2a} \pnode(-2.5,1){m2b} \pnode(-3.5,1){m3a} \pnode(-3.5,1){m3b} \edge{1a}{m1b} \edge{m1a}{3b} \edge{2a}{4b} \edge{3a}{m2b} \edge{m2a}{1b} \edge{4a}{m3b} \edge{m3a}{2b} }
\end{pspicture}
\end{center}

\end{example}

\section{Crossings and alignments}
\label{cro_ali}

For a permutation $\sigma\in S_n$, an \emph{alignment} is an unordered pair of
two arcs in the pignose diagram of $\sigma$ which look like one of the following
figures:
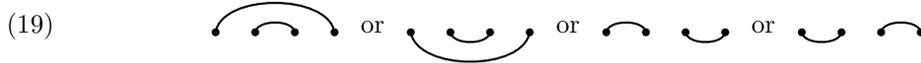
\begin{equation}
  \label{eq:6}
\raisebox{-15pt}{\begin{pspicture}(1,0)(4,2)
\vput1 \vput2 \vput3 \vput4
\edge14 \edge23
\end{pspicture}} \quad \mbox{or} \quad
\raisebox{-15pt}{\begin{pspicture}(1,0)(4,2)
\vput1 \vput2 \vput3 \vput4
\edge41 \edge32
\end{pspicture}} \quad \mbox{or} \quad
\raisebox{-15pt}{\begin{pspicture}(1,0)(4,2)
\vput1 \vput2 \vput3 \vput4
\edge12 \edge43
\end{pspicture}} \quad \mbox{or} \quad
\raisebox{-15pt}{\begin{pspicture}(1,0)(4,2)
\vput1 \vput2 \vput3 \vput4
\edge21 \edge34
\end{pspicture}}
\end{equation}

Let $\al(\sigma)$ denote the number of alignments of $\sigma$.
The following proposition was first proved by the first author \cite{Corteel2007} using
rather technical calculation. Here we provide another proof which is more
combinatorial.

\begin{prop}\label{prop:corteel}
If $\sigma\in S_n$ has $k$ weak excedances, then
\[ \cro(\sigma) + \al(\sigma) = (k-1)(n-k).\]
\end{prop}
\begin{proof}
  Since $\wex(\pi)=k$, we have $k$ upper arcs and $n-k$ lower arcs in the
  pignose diagram of $\pi$. 

  Let $A$ be the set of pairs $(U,L)$ of an upper arc $U$ and a lower arc
  $L$. Then there are $k(n-k)$ elements in $A$. We define $A_1$, $A_2$ and $A_3$
  to be the subsets of $A$ consisting of the pairs of arcs whose relative
  locations look like the following:
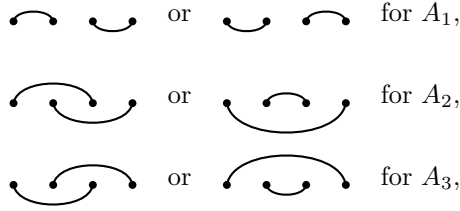
\begin{center}
\raisebox{-15pt}{\begin{pspicture}(1,0)(4,2)
\vput1 \vput2 \vput3 \vput4
\edge12 \edge43
\end{pspicture}} \quad or \quad
\raisebox{-15pt}{\begin{pspicture}(1,0)(4,2)
\vput1 \vput2 \vput3 \vput4
\edge21 \edge34
\end{pspicture}} \quad for $A_1$,

\raisebox{-15pt}{\begin{pspicture}(1,0)(4,2)
\vput1 \vput2 \vput3 \vput4
\edge13 \edge42
\end{pspicture}} \quad or \quad
\raisebox{-15pt}{\begin{pspicture}(1,0)(4,2)
\vput1 \vput2 \vput3 \vput4
\edge23 \edge41
\end{pspicture}} \quad for $A_2$,

\raisebox{-15pt}{\begin{pspicture}(1,0)(4,2)
\vput1 \vput2 \vput3 \vput4
\edge31 \edge24
\end{pspicture}} \quad or \quad
\raisebox{-15pt}{\begin{pspicture}(1,0)(4,2)
\vput1 \vput2 \vput3 \vput4
\edge14 \edge32
\end{pspicture}} \quad for $A_3$,
\end{center}

Observe that $A=A_1\uplus A_2 \uplus A_3$. Fix an upper arc $U$ whose right
endpoint is the second vertex of the $i$th pignose. Then by
Lemma~\ref{lem:divide}, the number of elements $(U,L)\in A_2$ is equal to the
number of pairs $(U,U')$ of upper arcs whose relative locations are the
following:
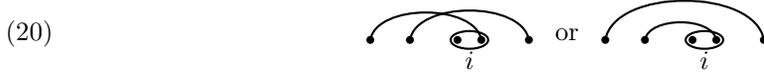
\begin{equation}
  \label{eq:4}
\raisebox{-15pt}{\begin{pspicture}(1,0)(5,2)
\vput1 \vput2 \vvput{3.2}{3} \vvput{3.8}{4} \vput5 \pig{3.5}
\rput(3.5,.5){$i$}
\edge14 \edge25
\end{pspicture}} \quad \mbox{or}  \quad
\raisebox{-15pt}{\begin{pspicture}(1,0)(5,2)
\vput1 \vput2 \vvput{3.2}{3} \vvput{3.8}{4} \vput5 \pig{3.5}
\rput(3.5,.5){$i$}
\edge24 \edge15
\end{pspicture}}
\end{equation}

Now fix a lower arc $L$ whose right endpoint is the first vertex of the $j$th
pignose. Again by Lemma~\ref{lem:divide}, the number of elements $(U,L)\in A_3$
is one more than the number of pairs $(L,L')$ of lower arcs whose relative
locations are the following:
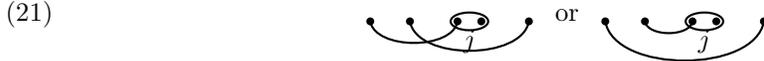
\begin{equation}
  \label{eq:5}
\raisebox{-15pt}{\begin{pspicture}(1,0)(5,2)
\vput1 \vput2 \vvput{3.2}{3} \vvput{3.8}{4} \vput5 \pig{3.5}
\rput(3.5,.5){$j$}
\edge31 \edge52
\end{pspicture}} \quad \mbox{or}  \quad
\raisebox{-15pt}{\begin{pspicture}(1,0)(5,2)
\vput1 \vput2 \vvput{3.2}{3} \vvput{3.8}{4} \vput5 \pig{3.5}
\rput(3.5,.5){$j$}
\edge51 \edge32
\end{pspicture}}
\end{equation}

Observe that a crossing or an alignment is either an element in $A_1$, or a pair
of arcs as shown in \eqref{eq:4} or \eqref{eq:5}.  Since we have $n-k$ lower
arcs, we obtain that the total number of crossings and alignments is equal to
\[
|A_1| + |A_2| + |A_3| - (n-k) = k(n-k)-(n-k) = (k-1)(n-k).
\]
\end{proof}

Now we define another representation of a signed permutation.  Note that a
signed permutation $\pi=\pi_1\cdots\pi_n\in B_n$ can be considered as a
bijection on $[\pm n]=\{1,2,\dots,n,-1,-2,\dots,-n\}$ with $\pi(i) = \pi_i$ and
$\pi(-i) = -\pi_i$ for $i\in[n]$.

We define the \emph{full pignose diagram} of $\pi\in B_n$ as follows. First we
arrange $2n$ pignoses in a horizontal line which are labeled with
$-n,-(n-1),\ldots,-1,1,2,\ldots, n$.  The \emph{first vertex} and the
\emph{second vertex} of a pignose labeled $i$ are, respectively, defined to be
the left vertex and the right vertex of the pignose if $i>0$, and to be the
right vertex and the left vertex of the pignose if $i<0$.  For each $i\in[\pm
n]$, we connect the first vertex of the pignose labeled with $i$ and the second
vertex of the pignose labeled with $\pi(i)$ with an arc in the following way. If
$\pi_i>0$, draw an arc above the horizontal line if $\pi(i)\ge i$, and below the
horizontal line if $\pi(i)<i$. For example, the following is the full pignose
diagram of $\pi=[4, -6, 1, -5, -3, 7, 2]$.

\[\psset{unit=.5cm}
\begin{pspicture}(-15,-3)(13,5) 
\vvput{-0.3}{1a} \vvput{0.3}{1b} \rput(0,.5){$1$} \pig{0} \vvput{1.7}{2a} \vvput{2.3}{2b} \rput(2,.5){$2$} \pig{2} \vvput{3.7}{3a} \vvput{4.3}{3b} \rput(4,.5){$3$} \pig{4} \vvput{5.7}{4a} \vvput{6.3}{4b} \rput(6,.5){$4$} \pig{6} \vvput{7.7}{5a} \vvput{8.3}{5b} \rput(8,.5){$5$} \pig{8} \vvput{9.7}{6a} \vvput{10.3}{6b} \rput(10,.5){$6$} \pig{10} \vvput{11.7}{7a} \vvput{12.3}{7b} \rput(12,.5){$7$} \pig{12} \vvput{-2.3}{m1a} \vvput{-1.7}{m1b} \rput(-2,1.5){$-1$} \pig{-2} \vvput{-4.3}{m2a} \vvput{-3.7}{m2b} \rput(-4,1.5){$-2$} \pig{-4} \vvput{-6.3}{m3a} \vvput{-5.7}{m3b} \rput(-6,1.5){$-3$} \pig{-6} \vvput{-8.3}{m4a} \vvput{-7.7}{m4b} \rput(-8,1.5){$-4$} \pig{-8} \vvput{-10.3}{m5a} \vvput{-9.7}{m5b} \rput(-10,1.5){$-5$} \pig{-10} \vvput{-12.3}{m6a} \vvput{-11.7}{m6b} \rput(-12,1.5){$-6$} \pig{-12} \vvput{-14.3}{m7a} \vvput{-13.7}{m7b} \rput(-14,1.5){$-7$} \pig{-14} \edge{1a}{4b} \edge{m1b}{m4a} \edge{m2b}{6b} \edge{2a}{m6a} \edge{3a}{1b} \edge{m3b}{m1a} \edge{m4b}{5b} \edge{4a}{m!
 5a} \edge{m5b}{3b} \edge{5a}{m3a} \edge{6a}{7b} \edge{m6b}{m7a} \edge{7a}{2b} \edge{m7b}{m2a} 
\end{pspicture}
\]

The following lemma can be shown by the same argument in the proof of
Lemma~\ref{lem:divide}. 

\begin{lem}\label{lem:divide_B}
  Let $\pi\in B_n$ and $k\in[n]$. In the full pignose diagram of $\pi$ if we
  draw a vertical line dividing the pignose labeled $-k$ (resp.~$k$) in the
  middle, then the number of upper arcs intersecting with this line is equal to
  the number of lower arcs intersecting with this line minus $1$ (resp.~plus
  $1$) as shown below.
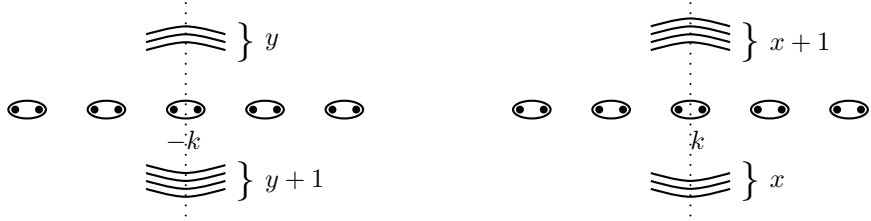
\[
\begin{pspicture}(-1,-2)(9,4) 
\vvput{-0.3}{1a} \vvput{0.3}{1b} \pig{0}
  \vvput{1.7}{2a} \vvput{2.3}{2b} \pig{2} \vvput{3.7}{3a} \vvput{4.3}{3b}
  \pig{4} \vvput{5.7}{4a} \vvput{6.3}{4b} \pig{6} \vvput{7.7}{5a}
  \vvput{8.3}{5b}\pig{8} 
\psline[linestyle=dotted](4,-1.7)(4,3.7)
\rput(4,2.5){\multido{\n=0+2}{3}{\rput(0,0.\n){\pscurve(-1,0)(0,.2)(1,0)}}}
\rput(4,-1){\multido{\n=0+2}{4}{\rput(0,0.\n){\pscurve(-1,0)(0,-.2)(1,0)}}}
\rput(5.5,2.7){\psscaleboxto(.5,1){\}}} \rput[l](6,2.7){$y$}
\rput(5.5,-.8){\psscaleboxto(.5,1){\}}} \rput[l](6,-.8){$y+1$}
\put(3.5,0){$-k$}
\end{pspicture}\qquad\qquad
\begin{pspicture}(-1,-2)(9,4) 
\vvput{-0.3}{1a} \vvput{0.3}{1b} \pig{0}
  \vvput{1.7}{2a} \vvput{2.3}{2b} \pig{2} \vvput{3.7}{3a} \vvput{4.3}{3b}
  \pig{4} \vvput{5.7}{4a} \vvput{6.3}{4b} \pig{6} \vvput{7.7}{5a}
  \vvput{8.3}{5b}\pig{8} 
\psline[linestyle=dotted](4,-1.7)(4,3.7)
\rput(4,2.5){\multido{\n=0+2}{4}{\rput(0,0.\n){\pscurve(-1,0)(0,.2)(1,0)}}}
\rput(4,-1){\multido{\n=0+2}{3}{\rput(0,0.\n){\pscurve(-1,0)(0,-.2)(1,0)}}}
\rput(5.5,2.7){\psscaleboxto(.5,1){\}}} \rput[l](6,2.7){$x+1$}
\rput(5.5,-.8){\psscaleboxto(.5,1){\}}} \rput[l](6,-.8){$x$}
\put(4,0){$k$}
\end{pspicture}
\]
\end{lem}

In a full pignose diagram, a \emph{positive half arc} (resp.~\emph{negative half
  arc}) is a half arc attached to a vertex of a pignose labeled with a positive
(resp.~negative) integer.  Note that the number of two arcs intersecting with
each other in the full pignose diagram of $\pi\in B_n$ is equal to $2\cro(\pi)$.

An \emph{alignment} of $\pi\in B_n$ is an unordered pair of arcs in the full
pignose diagram of $\pi$ whose relative locations are as shown in
\eqref{eq:6}. We denote by $\al(\pi)$ the number of alignments of $\pi$.

\begin{prop}\label{prop:corteelB}
For $\pi\in B_n$ with $\fwex(\pi)=k$, we have
\[ 2\cro(\pi) + \al(\pi) = n^2-2n+k.\]
\end{prop}
\begin{proof}
  Since the number of positive upper half arcs is equal to $\fwex(\pi)=k$, the
  number of positive lower half arcs is equal to $2n-k$.

Let $A$ be the set of pairs $(U,L)$ of an upper arc $U$ and a lower arc $L$.
Since we have $n$ upper arcs and $n$ lower arcs in total, there are $n^2$
elements in $A$. Using Lemma~\ref{lem:divide_B} and the same argument as in the
proof of Proposition~\ref{prop:corteel}, one can easily see that
\[ 2\cro(\pi)+\al(\pi) = n^2 -a - b,\]
where $a$ is the number of negative upper half arcs of the form
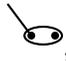
\[
\begin{pspicture}(0.5,.5)(2,2)
\vvput{1.2}{1} \vvput{1.8}{1}\pig{1.5}
\rput(-.6,0){\psline(1.8,1)(1.2,1.8)}
\end{pspicture},
\]
and $b$ is the number of positive lower half arcs of the form
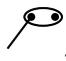
\[
\begin{pspicture}(0.5,0)(2,1.5)
\vvput{1.2}{1} \vvput{1.8}{1}\pig{1.5}
\rput(-.6,0){\psline(1.8,1)(1.2,0.2)}
\end{pspicture}.
\]
By the symmetry $a$ is equal to the number of positive lower half arcs of the
following form. 
\[
\begin{pspicture}(1,0)(2.5,1.5)
\vvput{1.2}{1} \vvput{1.8}{1}\pig{1.5}
\psline(1.8,1)(2.4,0.2)
\end{pspicture}
\]
Thus $a+b$ is the number of positive half arcs, which is $2n-k$, and we obtain
the desired formula.
\end{proof}

\section{Solutions of the Matrix Ansatz}
\label{MA}

Although it might be easy to check that some matrices satisfy the Matrix Ansatz \eqref{ansatz}, it is not obvious how to find 
explicitly such matrices.
We provide two solutions in the form of semi-infinite tridiagonal matrices. 
They can be obtained using the following observation: if $X$ and $Y$ are such that $XY-qYX=I$ (where $I$ is the identity), 
then $D=X(I+Y)$ and $E=YX(I+Y)$ satisfy 
\begin{align*}
DE-qED &=X(I+Y)YX(I+Y)-qYX(I+Y)X(I+Y) \\
       &=(XY+XYY-qYX-qYXY)X(I+Y) = (I+Y)X(I+Y) = D+E.
\end{align*}
Then we can look for $\bra{W}$ (respectively, $\ket{V}$) as a left (respectively, right) eigenvector
of $ytD-E$ (respectively, $D$).

\subsection{Solution 1}
Let $X=(X_{i,j})_{i,j\geq0}$ and $Y=(Y_{i,j})_{i,j\geq0}$ where 
\begin{align*}
  X_{i,i+1}=[i+1]_q \text{ and } X_{i,j}=0 \;\text{ otherwise,} \qquad 
  Y_{i+1,i}=1, \; Y_{i,i}=tyq^i    \text{ and } Y_{i,j}=0 \;\text{ otherwise,}
\end{align*}
and
\[ 
  \bra{W}=(1,0,0,\ldots),\qquad \ket{V}=(1,0,0,\ldots )^{\tr}.
\]
We can check that $XY-qYX=I$, and that $D=X(I+Y)$ and $E=YX(I+Y)$ together with $\bra{W}$ and $\ket{V}$ 
provide a solution of \eqref{ansatz}. The coefficients are:
\begin{align}
D_{i,i}= [i+1]_q, \qquad D_{i,i+1} = (1+ytq^{i+1})[i+1]_q, \qquad D_{i,j}=0 \text{ otherwise,}
\end{align}
and
\begin{align}
E_{i,i}= (1+ytq^i)[i]_q  +  ytq^i[i+1]_q, \quad E_{i,i+1} = ytq^i(1+ytq^{i+1})[i+1]_q, \quad E_{i+1,i}= [i+1]_q,
\end{align}
and $E_{i,j}=0$ otherwise.

Then, $y^2D+E$ can be seen as a transfer matrix for ``walks'' in the nonnegative integers. 
Since the matrix is tridiagonal and because of the particular choice of $\bra{W}$ and $\ket{V}$,
this shows that $B_n(y,t,q)$ count some weighted Motzkin paths of length $n$ (see the next section for 
more on the combinatorics of these paths). 
By standard methods, this gives a continued fraction for the generating function:

\begin{thm}[Theorem~\ref{bnfrac} in the introduction] Let
$\gamma_h  = y^2[h+1]_q + [h]_q + ytq^h([h]_q+[h+1]_q)$ and 
$\lambda_h = y[h]_q^2(y+tq^{h-1})(1+ytq^h)$, then we have:

\[
  \sum_{n\geq0} B_n(y,t,q) z^n = 
  \frac{1}{1-\gamma_0z} \cmo \frac{\lambda_1z^2}{1-\gamma_1z} \cmo \frac{\lambda_2z^2}{1-\gamma_2z} \cmo \cdots.
\]
\end{thm}

\begin{proof}
 It suffices to check $\gamma_h=y^2D_{h,h}+E_{h,h}$ and $\lambda_h=(y^2D_{h,h+1}+E_{h,h+1})E_{h+1,h}$.
\end{proof}

In particular, we have
\begin{align*}
  \sum_{n\geq0} B_n(y,1,-1) z^n = \frac{1}{1-(y+y^2)z- \frac{1-y(y+1)(1-y)z^2}{1-(1-y)z}}
   = \frac{1-z+zy}{1-z-y^2z}
   = 1+ \frac{(y+y^2)z}{1-(1+y)z}.
\end{align*}
It is then easy to obtain 
\[
  B_n(y,1,-1) = \sum_{k=1}^{2n} y^k \binom{n-1}{\lceil k/2 \rceil -1}.
\]

Using \eqref{eq:7}, we thus obtain as mentioned in Theorem~\ref{qeulrianb} that the value at $q=-1$ of the $q$-Eulerian numbers 
of type $B$ are binomial coefficients:
\begin{equation}
   E_{n,k}^B(-1) = [y^{2k}]B_n(y,1,-1) + [y^{2k+1}]B_n(y,1,-1) = \binom{n-1}{k-1} + \binom{n-1}{k} = \binom{n}{k}.
\end{equation}

\subsection{Solution 2}
Let $X=(X_{i,j})_{i,j\geq0}$ and $Y=(Y_{i,j})_{i,j\geq0}$ where 
\begin{align*}
  X_{i,i+1}=[i+1]_q  \text{ and } X_{i,j}=0 \text{ otherwise,} \qquad 
  Y_{i+1,i}=1   \text{ and } Y_{i,j}=0 \text{ otherwise,}
\end{align*}
and
\[ 
  \bra{W}=(1,yt,(yt)^2,\ldots),\qquad \ket{V}=(1,0,0,\ldots )^{\tr}.
\]
We can check that $XY-qYX=I$, and that $D=X(I+Y)$ and $E=YX(I+Y)$ together with $\bra{W}$ and $\ket{V}$ provide a solution 
of \eqref{ansatz}. The coefficients are:
\[
D_{i,i}=D_{i,i+1}=[i+1]_q,\qquad D_{i,j}=0\ {\rm otherwise},
\]
\begin{align*}
E_{i,i-1}=E_{i,i}=[i]_q, \qquad E_{i,j}=0\ {\rm otherwise}.
\end{align*}

\begin{thm}[Theorem~\ref{bnrec} in the introduction] We have $B_0=1$, and the recurrence relation:
\begin{equation} \label{An_rec}
   B_{n+1} (y,t,q) =  (y+t) D_q\big[   (1+yt) B_{n}(y,t,q) \big],
\end{equation}
where $D_q$ is the $q$-derivative with respect to $t$, which sends $t^n$ to $[n]_qt^{n-1}$.
\end{thm}

\begin{proof}
The choice of $\bra{W}$ leads us to consider a basis for polynomials in $t$ which is $1,yt,(yt)^2,\dots$
Then $v\mapsto \bra{W}v$ is just the realization of the identification between column vectors and polynomials in $t$.
In this basis, 
$Y$ is the matrix of multiplication by $yt$, and $X$ is the matrix of $y^{-1} D_q$. We have:
\[
  B_n(y,t,q) = \bra{W}(y^2D+E)^n \ket{V} = \bra{W}\big( (y^2I+Y)X(I+Y) \big)^n \ket{V}
\]
and the result follows. 
\end{proof}

In this case, seeing $y^2D+E$ as a transfer matrix shows that $B_n(y,t,q)$ counts some Motzkin suffixes. 

\begin{defn}
A {\it Motzkin suffix} of length $n$ and starting height $k$ is a path in $\mathbb{N}^2$ from 
$(0,k)$ to $(n,0)$ with steps $(1,1)$, $(1,0)$ and $(1,-1)$, respectively denoted $\nearrow$, $\rightarrow$ and $\searrow$.
The Motzkin {\it paths} are the particular cases where the starting height is 0. We denote by $\sh(p)$ the starting height
of the path $p$.
\end{defn}

Then, expanding $\bra{W}(y^2D+E)^n\ket{V}$ shows that $B_n(y,t,q)$ is the generating function of Motzkin suffixes $p$ with a weight
$(yt)^{\sh(p)}$, and weights $y^2D_{h,h+1}$ for a step $\nearrow$ from height $h$ to $h+1$, $y^2D_{h,h}+E_{h,h}$ for a step 
$\rightarrow$ at height $h$, and $E_{h+1,h}$ for a step $\searrow$ from height $h+1$ to $h$
(see the next section for more on the combinatorics of these paths). 

\section{\texorpdfstring{Interpretation of $B_n(y,t,q)$ via labeled Motzkin paths}{Interpretation of Bn(y,t,q) via labeled Motzkin paths}}
\label{paths}

We show here that some known bijections from \cite{Corteel2007} can be adapted to the case of signed
permutations. In this reference, the first author obtains refinements of two bijections originally given by
Fran\c con and Viennot \cite{FV}, Foata and Zeilberger \cite{FZ}. In the case of signed permutations, we will see that each of these
two bijections has two variants, corresponding to the two kinds of paths obtained in the previous section.

The two variants of the Fran\c con-Viennot bijection give new interpretations $B_n(y,t,q)$ in terms ascent-like and
pattern-like statistics in signed permutations. On the other side, 
the two variants of the Foata-Zeilberger bijection permit to recover the known interpretation as in \eqref{bnfwexcro}
(so we omit details in this case).

\subsection{Labeled Motzkin paths}

Let $\mathcal{M}_n$ be the set of weighted Motzkin paths of length $n$, where each step is either:
\begin{itemize}
 \item $\nearrow$ from height $h$ to $h+1$ with weight $q^i$, $0\leq i\leq h$ (type 1),
 \item $\nearrow$ from height $h$ to $h+1$ with weight $ytq^{h+1+i}$, $0\leq i\leq h$ (type 2),
 \item $\rightarrow$ at height $h$ with weight $y^2q^i$, $0\leq i\leq h$ (type 3),
 \item $\rightarrow$ at height $h$ with weight $ytq^{h+i}$, $0\leq i\leq h$ (type 4),
 \item $\rightarrow$ at height $h$ with weight $q^{i}$, $0\leq i\leq h-1$ (type 5),
 \item $\rightarrow$ at height $h$ with weight $ytq^{h+i}$, $0\leq i\leq h-1$ (type 6),
 \item $\searrow$ from height $h+1$ to $h$ with weight $y^2q^i$, $0\leq i\leq h$ (type 7),
 \item $\searrow$ from height $h+1$ to $h$ with weight $ytq^{h+i}$, $0\leq i\leq h$ (type 8).
\end{itemize}
This set has generating function $B_n(y,t,q)$ since the weights correspond to coefficients of $y^2D+E$,
where we use the first solution of the Matrix Ansatz from the previous section. More precisely, the correspondence is: 
$E_{h,h-1}\to$ types 1 and 2, $D_{h,h}\to$ type 3, $E_{h,h}\to$ types 4, 5 and 6, 
$D_{h,h+1}\to$ type 7, $E_{h,h+1}\to$ type 8.
Note that the weight on a
step does not always determine its type (compare type 4 and 6), so that we need to think in terms of
labeled paths where each step has a label between 1 and 8 to indicate its type.

\subsubsection{The Fran\c con-Viennot bijection, first variant}

There is a bijection between $S_n$ and weighted Motzkin paths which follows the
number of descents and the number of 31-2 patterns in a permutation
\cite{Corteel2007}. The path is obtained by ``scanning'' the graph of a
permutation from bottom to top, {\it i.e.} the $i$th step is obtained by
examining $\sigma^{-1}(i)$.  In the case of a signed permutation $\pi$, we use
the representation as in Figure~\ref{fv1}: we place a $+$ or a $-$ in the square
in the $i$th column and $|\pi_i|$th row depending on the sign of $\pi(i)$. 

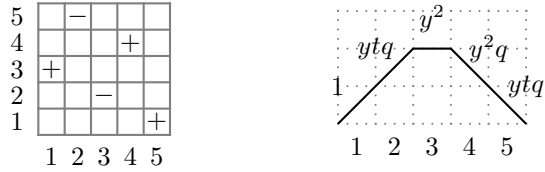
\begin{figure}[h!tp]\psset{unit=3.5mm}
\begin{pspicture}(-1,-1)(5,5)
\psgrid[gridcolor=gray,griddots=0,subgriddiv=0,gridlabels=0](0,0)(5,5)
\rput(0.5,2.5){$+$}\rput(1.5,4.5){$-$}\rput(2.5,1.5){$-$}\rput(3.5,3.5){$+$}\rput(4.5,0.5){$+$}
\rput(-0.8,0.5){1}
\rput(-0.8,1.5){2}
\rput(-0.8,2.5){3}
\rput(-0.8,3.5){4}
\rput(-0.8,4.5){5}
\rput(0.5,-0.8){1}
\rput(1.5,-0.8){2}
\rput(2.5,-0.8){3}
\rput(3.5,-0.8){4}
\rput(4.5,-0.8){5}
\end{pspicture}
\hspace{2cm} \psset{unit=5mm}
\begin{pspicture}(0,-1)(5,3)
   \psgrid[gridcolor=gray,griddots=4,subgriddiv=0,gridlabels=0](0,0)(5,3)
   \psline(0,0)(2,2)(3,2)(5,0)
  \rput(0.5,-0.6){1} \rput(1.5,-0.6){2} \rput(2.5,-0.6){3} \rput(3.5,-0.6){4} \rput(4.5,-0.6){5}
  \rput(0,1){$1$} \rput(1,2){$ytq $} \rput(2.5,2.8){$y^2$} \rput(4,2){$y^2q$} \rput(5,1){$ytq$}
\end{pspicture}
\caption{The first variant the Fran\c con-Viennot bijection, with the signed
  permutation $\pi=3,-5,-2,4,1$. }
\label{fv1}
\end{figure}

\begin{defn} For any $\sigma\in S_n$ and $1\leq i\leq n$,
   let $\mot(\sigma,i) = \#\{ \; j \;  : \;  1\leq j<i-1 \hbox{ and } \pi_j>\pi_i>\pi_{j+1}  \}$, and let 
   $\mott(\sigma,i) = \#\{ \; j \;  : \;  i<j<n \hbox{ and } \pi_j>\pi_i>\pi_{j+1}  \}$.
   Let also $\mot(\sigma) = \sum_{i=1}^n \mot(\sigma,i) $.
\end{defn}

Let $\pi\in B_n$. We denote by $|\pi|$ the permutation obtained by removing the
negative signs of $\pi$, that is, $|\pi|_i = |\pi_i|$ for $i=1,2,\dots,n$.  We
take the convention that $\pi_0=0$ and $\pi_{n+1}=n+1$.  The bijection is
defined in the following way. The path corresponding to $\pi$ is of length $n$
such that, if $j=|\pi_i|$ and denoting $s$ the $j$th step, we have:
\begin{itemize}
 \item If $|\pi_{i-1}|>|\pi_i|<|\pi_{i+1}|$ and $\pi_i>0$, then $s$ is of type 1 with weight $q^{\mot(|\pi|,i) }$.
 \item If $|\pi_{i-1}|>|\pi_i|<|\pi_{i+1}|$ and $\pi_i<0$, then $s$ is of type 2 with weight $ytq^{h+1+\mot(|\pi|,i) }$.
 \item If $|\pi_{i-1}|<|\pi_i|<|\pi_{i+1}|$ and $\pi_i>0$, then $s$ is of type 3 with weight $y^2q^{\mot(|\pi|,i) }$.
 \item If $|\pi_{i-1}|<|\pi_i|<|\pi_{i+1}|$ and $\pi_i<0$, then $s$ is of type 4 with weight $ytq^{h+\mot(|\pi|,i) }$.
 \item If $|\pi_{i-1}|>|\pi_i|>|\pi_{i+1}|$ and $\pi_i>0$, then $s$ is of type 5 with weight $q^{\mot(|\pi|,i) }$.
 \item If $|\pi_{i-1}|>|\pi_i|>|\pi_{i+1}|$ and $\pi_i<0$, then $s$ is of type 6 with weight $ytq^{h+\mot(|\pi|,i) }$.
 \item If $|\pi_{i-1}|<|\pi_i|>|\pi_{i+1}|$ and $\pi_i>0$, then $s$ is of type 7 with weight $y^2q^{\mot(|\pi|,i) }$.
 \item If $|\pi_{i-1}|<|\pi_i|>|\pi_{i+1}|$ and $\pi_i<0$, then $s$ is of type 8 with weight $ytq^{h+\mot(|\pi|,i) }$.
\end{itemize}

In the case where $\pi$ has no negative entry, this defines a bijection with the
paths having steps of types 1, 3, 5, 7 only.
This is a result from \cite{Corteel2007}, and what we present here is a variant, so we refer to this work for
more details. In the case of signed permutations, since each entry can be negated, it is natural that
each step of type (1, 3, 5, or 7) has a respective variant (type 2, 4, 6, or 8). So the fact that the map is a bijection can
be deduced from the case of unsigned permutations. It is also clear that $t$ follows the number of negative entries
of $\pi$ since there is a factor $t$ only in steps of type 2, 4, 6, and 8.

Let us check what is the statistic followed by $y$. There is a factor $y$ on each step of type 2,4,6,8 so that this 
statistic is the sum of $\neg(\pi)$ and other terms coming from the factors $y^2$ in steps of type 
3 and 7. Note that the $j$th step is of type 3 or 7 if and only if $|\pi_{i-1}|<\pi_i$. This leads us to define an ascent
statistic $\hasc(\pi)$ as
\begin{equation} \label{defhasc}
   \hasc(\pi) = 2 \times \# \{\; i \; : \;   0 \leq i \leq n-1 \hbox{ and } |\pi_i| < \pi_{i+1}    \} + \neg(\pi).
\end{equation}

It remains only to check what is the statistics followed by $q$. Since there is always a weight $q^{\mot(|\pi|,i)}$
on the $i$th step, this statistic is the sum of $\mot(|\pi|)$ and other terms. It remains to take into account the weights
$q^h$ or $q^{h+1}$ that appear on each step of type 2,4,6,8. From the properties of the bijection in the unsigned case,
we have that the ``minimal'' height $h$ of the $i$th step is $\mot(|\pi|,i)+\mott(|\pi|,i)$, plus 1 in the case where
$|\pi_{i-1}|>|\pi_i|>|\pi_{i+1}|$. So, apart a factor $q$ on each step of type 2 and 6, we obtain the statistic
\[
   \mot(|\pi|) + \sum_{\substack{ 1\leq i \leq n \\ \pi_i<0 } } \big(  \mot(|\pi|,i) + \mott(|\pi|,i)    \big).
\]
Note that to take into account the factor $q$ on each step of type 2 and 6, we can count $i$ such that
$1\leq i<n$ and $|\pi_{i-1}|>-\pi_i>0 $. The statistic we eventually obtain can be rearranged as follows:
\begin{align*}
 \pat(\pi) = \; &  \# \{ \; (i,j) \; : \;   1\leq i<j \leq n, \hbox{ and } |\pi_i|>|\pi_j|>|\pi_{i+1}| \} \; + \\
            &  \# \{\; (i,j) \; : \;   1\leq i,j \leq n, \hbox{ and } |\pi_i|>-\pi_j \geq |\pi_{i+1}| \}.
\end{align*}
This is a ``pattern'' statistic that is somewhat similar to the $\mot$ statistic of the unsigned case
(and indeed identical if $\pi$ has only positive entries). Eventually, our first variant of the Fran\c con-Viennot
bijection gives the following.

\begin{prop}[First equality of Theorem~\ref{bndes} in the introduction]
\begin{equation}
 B_n(y,t,q) = \sum_{\pi \in B_n}  y^{\hasc(\pi)}    t^{\neg(\pi)} q^{ \pat (\pi) }.
\end{equation}
\end{prop}

Note that in the case $t=0$, we recover the known result:
\begin{equation} \label{enkmot}
 \sum_{k=1}^n y^k E_{n,k}(q) = \sum_{\sigma \in S_n} y^{\asc(\sigma)+1} q^{\mot(\sigma)}
\end{equation}
where $\asc(\sigma)=n-1-\des(\sigma)=\#\{ i : 1\leq i \leq n-1 \text{ and } \sigma_i<\sigma_{i+1}\}$.

\subsubsection{The Foata-Zeilberger bijection, first variant}

In case of unsigned permutations, this bijection follows the number of weak excedances and crossings,
see the bijection $\Psi_{FZ}$ in \cite{Corteel2007}. To extend it, we use a representation of a signed 
permutation as an arrow diagram (this is equivalent to the pignose diagrams used earlier: if $\sigma\in S_n$,
we put $n$ dots in a row, and draw an arrow from $i$ to $\sigma(i)$ which is above the axis if $i\leq\sigma(i)$ and
below otherwise, see \cite{Corteel2007}).
The idea is to draw the arrow diagram of $|\pi|$ and label an arrow 
from $i$ to $|\pi(i)|$ by $+$ or $-$ depending on the sign of $\pi(i)$, see Figure~\ref{fz2}.

\begin{figure}[h!tp] \psset{unit=8mm}
 \begin{pspicture}(1,-1.6)(5,1)
  \psdots(1,0)(2,0)(3,0)(4,0)(5,0)
    \pscurve{->}(1,0)(2,0.7)(3,0)  \pscurve{->}(2,0)(3,0.7)(4,0) \pscurve{->}(4,0)(4.5,0.3)(5,0)
   \pscurve{->}(3,0)(2.5,-0.3)(2,0) \pscurve{->}(5,0)(3,-1)(1,0)
   \rput(2,1){$+$} \rput(3,1){$-$} \rput(4.5,0.8){$+$} \rput(2.5,-0.5){$-$} \rput(3,-1.4){$+$}
  \rput(1,-0.3){\small 1} \rput(2,-0.3){\small 2} \rput(3,-0.3){\small 3} \rput(4,-0.3){\small 4} \rput(5,-0.3){\small 5}   
 \end{pspicture}
\hspace{2cm}    \psset{unit=6mm}
\begin{pspicture}(0,-0.8)(5,3)
   \psgrid[gridcolor=gray,griddots=4,subgriddiv=0,gridlabels=0](0,0)(5,3)
   \psline(0,0)(2,2)(3,1)(4,1)(5,0)
  \rput(0.5,-0.6){1} \rput(1.5,-0.6){2} \rput(2.5,-0.6){3} \rput(3.5,-0.6){4} \rput(4.5,-0.6){5}
  \rput(-0.6,0.7){$(1,y^2)$} \rput(-0.2,2){$(ytq^3,ytq)$} \rput(2.5,2.2){$1$} \rput(3.7,1.5){$y^2q$} \rput(5,0.6){$1$}
\end{pspicture}
\caption{The first variant of the Foata-Zeilberger bijection, with
  $\pi=3,-4,-2,5,1$.}
\label{fz2}
\end{figure}
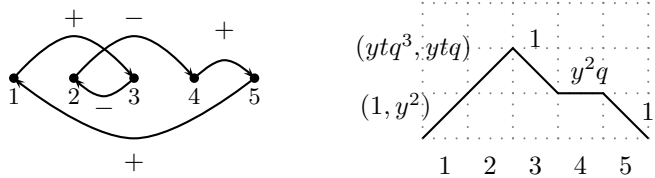

The crossings can be read in this representation as follows. The proof can be done by distinguishing all the possible
cases for the signs of $i$ and $j$ for each type of crossing $(i,j)$. We omit details.

\begin{prop}
 Each crossing $(i,j)$ in the signed permutation $\pi$ corresponds to one of the six
 configurations in Figure~\ref{confcross}, where $\pm$ means that the label of the arrow can be either $+$ or $-$,
 and the dots indicate that there might be an equality of the endpoint of an arrow and the startpoint of the other arrow.

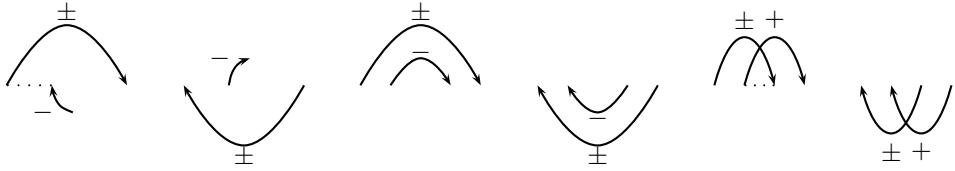
\begin{figure}[h!tp]
\psset{unit=4mm}
  \begin{pspicture}(0,-2.6)(4,2.9)
   \pscurve{->}(0,0)(2,2)(4,0)\rput(2,2.4){$\pm$} \pscurve{->}(2.2,-0.9)(1.8,-0.7)(1.6,-0.4)(1.5,0)
   \psline[linestyle=dotted,dotsep=1mm](0,0)(1.5,0)
   \rput(1.2,-0.9){$-$}
  \end{pspicture}
  \hspace{0.5cm}
  \begin{pspicture}(0,-2.6)(4,2.2)
   \pscurve{<-}(0,0)(2,-2)(4,0)\rput(2,-2.4){$\pm$} \pscurve{<-}(2.2,0.9)(1.8,0.7)(1.6,0.4)(1.5,0)
   \rput(1.2,0.9){$-$}
  \end{pspicture}
  \hspace{0.5cm}
  \begin{pspicture}(0,-2.6)(4,2.2)
   \pscurve{->}(0,0)(2,2)(4,0)\rput(2,2.4){$\pm$} \pscurve{->}(1,0)(2,0.9)(3,0)
   \rput(2,1.1){$-$}
  \end{pspicture}
  \hspace{0.5cm}
  \begin{pspicture}(0,-2.6)(4,2.2)
   \pscurve{<-}(0,0)(2,-2)(4,0)\rput(2,-2.4){$\pm$} \pscurve{<-}(1,0)(2,-0.9)(3,0)
   \rput(2,-1.1){$-$}
  \end{pspicture}
  \hspace{0.5cm}
  \begin{pspicture}(1,-2.6)(4,2.2)
   \pscurve{->}(1,0)(2,1.6)(3,0)\rput(2,2.1){$\pm$}\pscurve{->}(2,0)(3,1.6)(4,0)\rput(3,2.2){$+$}
   \psline[linestyle=dotted,dotsep=1mm](2,0)(3,0)
  \end{pspicture}
  \hspace{0.5cm}
  \begin{pspicture}(1,-2.6)(4,2.2)
   \pscurve{<-}(1,0)(2,-1.6)(3,0)\rput(2,-2.3){$\pm$}\pscurve{<-}(2,0)(3,-1.6)(4,0)\rput(3,-2.2){$+$}
  \end{pspicture}
\caption{ Configurations that characterize a crossing in the arrow diagram of a
  signed permutation. }
\label{confcross}
\end{figure}
\end{prop}

For example, let us consider the signed permutation in Figure~\ref{fz2}. The dots 1,2,3 (respectively, 1,2,5 and 1,2,3,5)
correspond to a crossing as in the first (respectively, second and fourth) configuration of Figure~\ref{confcross}. 
The dots 2,4 are also a crossing as in (the limit case of) the first configuration, and eventually the dots 2,4,5 are a 
crossing as in (the limit case of) the fifth
configuration. So the signed permutation $3,-4,-2,5,1$ has 5 crossings.

Once we have this new description of the crossings, it is possible to encode the arrow diagram as in Figure~\ref{fz2}
by elements in $\mathcal{M}_n$. 
Actually it will be more convenient to see these paths in a slightly different way: each step $\nearrow$ (with weight $a$) 
faces a step $\searrow$ (weight $b$), but now we think as if the step $\searrow$ have weight $1$ and the facing step $\nearrow$ 
has weight $(a,b)$. But we still consider the steps $\rightarrow$ of types 3, 4, 5, 6 as before.

Now, the path is obtained from the signed permutation by ``scanning'' the arrow diagram from right to left, so that after
scanning $i$ nodes in the diagram we have built a Motzkin suffix of length $i$. Suppose that there are $h$
unconnected strands after scanning these $i$ nodes, and accordingly the Motzkin suffix starts at height $h$.
When scanning the following node, we have several possibilities:
\begin{itemize}
\item If the $i$th node is $\nodg$ then we add a step $\searrow$ (with weight 1) to the Motzkin suffix,
\item If the $i$th node is $\nodfp$ or $\nodhp$, then we add a step $\rightarrow$ (type 3) with weight
        $y^2q^i$ where $i$ counts the number of crossings as in the 5th configuration that appear when adding this
        $i$th node to the ones at its right.
\item If the $i$th node is $\nodfm$ then we add a step $\rightarrow$ (type 4) with weight $yt$ to the Motzkin suffix,
        and if it is $\nodhm$ then we add a step $\rightarrow$ (type 4) with weight $ytq^{h+i}$, where $i$ counts the
        number of crossings as in 3rd configuration that appear when adding this $i$th node to the ones to its right.
        Note that $h$ is the number of crossings as in 2nd configuration that appear.
\item If the $i$th node is $\nodbp$, then we add a step step $\rightarrow$ (type 5) with weight
        $q^{i}$ where $i$ counts the number of crossings as in the 6th configuration that appear. 
\item If the $i$th node is $\nodbm$, then we add a step $\rightarrow$ (type 6) with weight
        $ytq^{h+i}$ where $i$ counts the number of crossings as in the 4th configuration that appear. Note that
        $h$ is the number of crossings as in 1st configuration that appear.
\item If the $i$th node is $\nodr$, then we add a step $\nearrow$ with a weight $(a,b)$ where $a$ and $b$ are as follows
        (this is similar to the previous cases but here we need to encode information about both the ingoing and outgoing arrow).
        The possibilities for $a$ are $q^0,\dots,q^{h-1},$ or $ytq^{h},\dots,ytq^{2h-1}$, such that there is a factor $1$
        (resp. $yt$) if the label of the ingoing arrow is $+$ (resp. $-$), and $q$ counts the crossings of the 6th (resp. 1st and 4th)
        configuration.
        The possibilities for $b$ are $y^2q^0,\dots,y^2q^{h-1},$ or $ytq^{h-1},\dots,ytq^{2h-2}$, such that there is a factor $y^2$
        (resp. $yt$) if the label of the outgoing arrow is $+$ (resp. $-$), and $q$ counts the crossings of the 5th (resp. 2nd and 3rd)
        configuration. 
\end{itemize}

See Figure~\ref{fz2} for an example.
Let $i<j$, if there is an arrow from $i$ to $j$ with a label $+$ then the $i$th step gets a weight $y^2$, and if there is an arrow
from $i$ to $j$ or from $j$ to $i$ with a label $-$, then the $i$th step gets a weight $yt$. It follows that the parameters $y$ and $t$
correspond to $\fwex(\pi)$ and $\neg(\pi)$ in signed permutations. By the design of the bijection, the parameter $q$
corresponds to $\cro(\pi)$. Hence we recover \eqref{bnfwexcro}.


\subsection{Suffixes of labeled Motzkin paths}

We can see a signed permutation as a permutation on $[\pm n]$ or on $[2n]$ with a symmetry property. 
Then, applying the bijections of the unsigned case gives some weighted Motzkin paths with a vertical symmetry.
It is natural to keep only the second half a vertically-symmetric path, and obtain suffixes of Motzkin paths.

\begin{defn}
Let $\mathcal{N}_n$ be the set of weighted Motzkin suffixes with weights:
\begin{itemize}
 \item either $y^2q^i$ with $0\leq i\leq h$, or $q^i$ with $0\leq i\leq h-1$, for a step $\rightarrow$ at height $h$,
 \item $y^2q^i$ with $0\leq i\leq h$, for a step $\nearrow$ at height $h$ to $h+1$,
 \item $q^i$ with $0\leq i\leq h$, for a step $\searrow$ at height $h+1$ to $h$.
\end{itemize}
For any $p\in\mathcal{N}_n$, let $\sh(p)$ be its initial height, and let $\w(p)$ be its total weight, {\it i.e.} the product of
the weight of each step.
\end{defn}
Then $B_n(y,t,q)$ is the generating function $\sum_{p\in \mathcal{N}_n} (yt)^{\sh(p)}\w(p)$, because the paths are the ones
arising from the second solution of the Matrix Ansatz in the previous section.

\subsubsection{The Fran\c con-Viennot bijection, second variant}

This second variant is a bijection between $\mathcal{N}_n$ and $B_n$, and it gives a combinatorial model of $B_n(y,t,q)$
involving the flag descents, and different from the previous ones. 
The second bijection is done using the diagram of a signed permutation as in Figure~\ref{fv2}.

\begin{figure}[h!tp] \psset{unit=4mm}
 \begin{pspicture}(-1,-1)(10,10)
\psgrid[gridcolor=gray,griddots=0,subgriddiv=0,gridlabels=0](0,0)(10,10)
\rput(5.5,7.5){$\bullet$}
\rput(6.5,0.5){$\bullet$}
\rput(7.5,3.5){$\bullet$}
\rput(8.5,8.5){$\bullet$}
\rput(9.5,5.5){$\bullet$}
\rput(0.5,4.5){$\bullet$}
\rput(1.5,1.5){$\bullet$}
\rput(2.5,6.5){$\bullet$}
\rput(3.5,9.5){$\bullet$}
\rput(4.5,2.5){$\bullet$}
\rput(-0.8,0.5){-5}
\rput(-0.8,1.5){-4}
\rput(-0.8,2.5){-3}
\rput(-0.8,3.5){-2}
\rput(-0.8,4.5){-1}
\rput(-0.8,5.5){1}
\rput(-0.8,6.5){2}
\rput(-0.8,7.5){3}
\rput(-0.8,8.5){4}
\rput(-0.8,9.5){5}
\rput(0.1,-0.8){-5}
\rput(1.2,-0.8){-4}
\rput(2.3,-0.8){-3}
\rput(3.4,-0.8){-2}
\rput(4.5,-0.8){-1}
\rput(5.5,-0.8){1}
\rput(6.5,-0.8){2}
\rput(7.5,-0.8){3}
\rput(8.5,-0.8){4}
\rput(9.5,-0.8){5}
\psline[linestyle=dashed,linecolor=black,linewidth=0.8mm](-0.5,5)(10.5,5)
\end{pspicture}
\hspace{2cm} \psset{unit=5mm}
 \begin{pspicture}(0,-2)(5,3)
   \psgrid[gridcolor=gray,griddots=4,subgriddiv=0,gridlabels=0](0,0)(5,4)
   \psline(5,0)(4,1)(3,2)(2,3)(1,3)(0,2)
   \rput(0.5,-0.6){1} \rput(1.5,-0.6){2} \rput(2.5,-0.6){3} \rput(3.5,-0.6){4} \rput(4.5,-0.6){5}
   \rput(0,3){$y^2q^2$} \rput(1.5,3.5){$y^2$} \rput(3,3){$q$} \rput(4,2){$q$} \rput(5,1){$1$}
 \end{pspicture}
\caption{The second variant of the Fran\c con-Viennot bijection, with
  $\pi=3,-5,-2,4,1$  }
\label{fv2}
\end{figure}
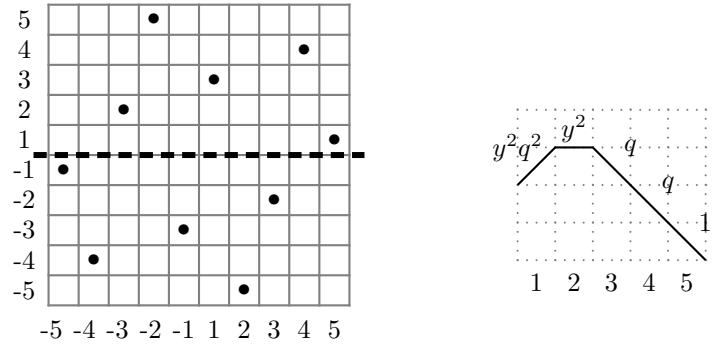


\begin{defn} For a signed permutation $\pi$, let
$\fneg(\pi)$ be the number of positive integers followed by a negative integer in the sequence
$\pi(-n),\dots,\pi(-1),\pi(1),\dots,\pi(n)$. Note that $\fneg(\pi)=0$ if and only if $\pi$ is actually an
unsigned permutation. Let
\[
    \mot^+(\pi) = \sum_{ \substack{ 1 \leq i \leq 2n  \\   \tilde\pi(i) > n   } } \mot( \tilde\pi , i )
\]
where $\tilde\pi \in S_{2n}$ is the permutation corresponding to $\pi$ via the 
order-preserving identification $[\pm n] \to \{1,\dots, 2n  \}$.
\end{defn}

The properties of the Fran\c con-Viennot bijection show the following.

\begin{prop}[Second equality of Theorem~\ref{bndes} in the introduction] For $n\geq1$,
\begin{equation}
 B_n(y,t,q) = \sum_{\pi \in B_n}  y^{\fdes(\pi) + 1 }    t^{\fneg(\pi)} q^{ \mot^+ (\pi) }.
\end{equation}
\end{prop}

Note that once again, we recover \eqref{enkmot} in the case $t=0$.

\subsubsection{The Foata-Zeilberger bijection, second variant}
Let $\pi\in B_n$, 
we use here the arrow diagram as in Figure~\ref{fz1} where the dots are labeled by $-n,\dots,-1,1,\dots,n$ and there is
an arrow from $i$ to $\pi(i)$ above the axis if $i\leq\pi(i)$ and below otherwise (this is equivalent to the full pignose
diagram used in Section~\ref{cro_ali} where each pignose collapses to a single dot).
The path $p\in\mathcal{N}_n$ is as follows.
If at the $i$th node in the diagram (where $1\leq i \leq n$), there is an arrow arriving from the left and an arrow going to
the left, then the $i$th step is $\searrow$. If there is an arrow going to the right and arriving from the right, then it is a step
$\nearrow$. In all other cases it is a step $\rightarrow$.
Then, for each arrow going from $i$ to $j$ with $0<i<j$, we give a weight $y^2$ to the $i$th step in the path;
for each crossing $i<j\leq\pi(i)<\pi(j)$ with $j>0$ we give a weight $q$ to the $j$th step; and
for each crossing $i>j>\pi(i)>\pi(j)$ with $j>0$ we give a weight $q$ to the $j$th step.
See Figure~\ref{fz1} for example.

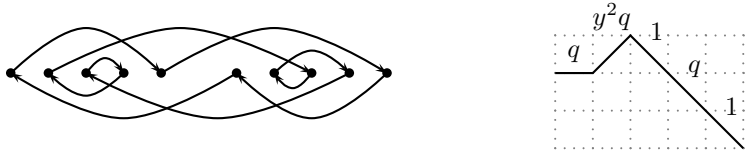
\begin{figure}[h!tp] \psset{unit=5mm}
 \begin{pspicture}(-5,-2)(5,2)
  \psdots(-5,0)(-4,0)(-3,0)(-2,0)(-1,0)
  \psdots(1,0)(2,0)(3,0)(4,0)(5,0)
   \pscurve{->}(2,0)(3,0.6)(4,0) \pscurve{->}(-2,0)(-3,-0.6)(-4,0)
   \pscurve{->}(3,0)(2.5,-0.4)(2,0) \pscurve{->}(-3,0)(-2.5,0.4)(-2,0)
  \pscurve{->}(5,0)(3,-1.2)(1,0) \pscurve{->}(-5,0)(-3,1.2)(-1,0)
   \pscurve{->}(1,0)(-2,-1.2)(-5,0)\pscurve{->}(-1,0)(2,1.2)(5,0)
  \pscurve{->}(4,0)(0.5,-1.2)(-3,0) \pscurve{->}(-4,0)(-0.5,1.2)(3,0)
 \end{pspicture}
\hspace{2cm}
 \begin{pspicture}(0,0)(5,3)
 \psgrid[gridcolor=gray,griddots=4,subgriddiv=0,gridlabels=0](0,0)(5,3)
 \psline(5,0)(4,1)(3,2)(2,3)(1,2)(0,2)
 \rput(0.5,2.5){$q$}\rput(1.5,3.5){$y^2q$}\rput(2.7,3.1){$1$}\rput(3.7,2.1){$q$}\rput(4.7,1.1){$1$}
 \end{pspicture}
\caption{The second variant of the Foata-Zeilberger bijection in the case of
  $\pi=-5,4,2,-3,1$.   }
\label{fz1}
\end{figure}
The number of arrows that join a positive integer to a negative one is $\neg(\pi)$.
From the construction, this is also the initial height of the path we have built.
So giving a weight $(yt)^{\sh(p)}$ to the path $p$ ensures that $y$ and $t$ respectively follow $\fwex(\pi)$ and $\neg(\pi)$.
Also from the definition of the bijection, $q$ counts the number of crossings.
Hence we recover \eqref{bnfwexcro}.

\section{Enumeration formula}
\label{EF}

%
%

In this section, we prove a formula for $B_n(y,1,q)$. The formula itself is
somewhat similar to the following one for $B_n(y,0,q)$ proved in \cite{J} but
the proof is different:
\begin{equation}
  \label{eq:y0q}
B_n(y,0,q) = \frac{1}{(1-q)^n} \sum_{k=0}^n 
\left(
\sum_{j=0}^{n-k} y^j 
  \left( \tbinom{n}{j}\tbinom{n}{j+k} - \tbinom{n}{j-1}\tbinom{n}{j+k+1} \right)
\right)
\left( \sum_{i=0}^k (-1)^k y^i q^{i(k+1-i)} \right),
\end{equation}

\begin{thm}[Theorem~\ref{bnformula} in the introduction]
\begin{equation} \label{formulean1}
  B_n(y,1,q) =
  \frac{1}{(1-q)^n} \sum_{j=0}^n (-1)^j
  \left(\sum_{i=0}^{2n-2j}  y^i \binom{n}{j+\lceil \frac i2 \rceil}
                                \binom{n}{\lfloor \frac i2 \rfloor} \right)
  \left(\sum_{\ell=0}^{2j}    y^\ell q^{\frac{\ell(2j-\ell+1)}{2}}          \right).
\end{equation}
Or, equivalently:
\begin{equation} \label{formulean2}
B_n(y,1,q) =  \frac 1{(1-q)^n} \sum\limits_{k=0}^n
 \left(\sum\limits_{i=0}^{n-k} y^{2i}
    \Big( \tbinom{n}{i}\tbinom{n}{i+k} -
   \tbinom{n}{i-1}\tbinom{n}{i+k+1}\Big) \right)
   \sum\limits_{j=0}^k y^{k-j} (-1)^j  \sum_{\ell=0}^{2j} y^\ell  q^{\frac{\ell(2j-\ell+1)}2}.
\end{equation}
\end{thm}

We can obtain \eqref{formulean1} from \eqref{formulean2} by simplifying a summation as follows. Let 
$P_j = \sum_{\ell=0}^{2j} y^\ell  q^{\frac{\ell(2j-\ell+1)}2}$, then the right-hand side of \eqref{formulean2} is
\begin{align*}
  \frac 1{(1-q)^n} \sum\limits_{j=0}^n (-1)^j P_j 
   \sum_{k=j}^n \left(\sum\limits_{i=0}^{n-k} y^{k-j+2i}
   \Big( \tbinom{n}{i}\tbinom{n}{i+k} -
   \tbinom{n}{i-1}\tbinom{n}{i+k+1}\Big) \right)  \\
= 
  \frac 1{(1-q)^n} \sum\limits_{j=0}^n (-1)^j P_j 
    \sum_{m=0}^{2n-2j} y^m \ \sum_{i=0}^{\lfloor m/2 \rfloor} \Big( \tbinom{n}{i}\tbinom{n}{m+j-i} - \tbinom{n}{i-1}\tbinom{n}{m+j-i+1} \Big).
\end{align*}
Here we have introduced the new index $m=2i+k-j$, and the condition $k-j\geq0$ gives the condition $m\geq 2i$, {\it i.e.} $i\leq\lfloor m/2 \rfloor $.
But the $i$-sum in the latter formula is actually telescopic and only the term $\binom ni \binom{n}{m+j-i} $ with $i=\lfloor m/2 \rfloor$ remains. Using the fact that
$m-\lfloor m/2\rfloor=\lceil m/2\rceil $, we obtain \eqref{formulean1}.

So, \eqref{formulean1} is a simpler formula, but \eqref{formulean2} is the one which is conveniently proved, using results
from \cite{JK}. The theorem follows from the two lemmas below (and a third lemma is needed to prove the second lemma).
The first lemma was essentially present in \cite{josuatrubey}.

\begin{lem}
 Suppose that two sequences $(b_n)_{n\ge0}$ and $(c_n) _{n\ge0}$ are such that:
\begin{equation} \label{fracan}
   \sum_ {n\geq0} b_n z^n = 
   \frac{1}{1-\gamma_0z} \cmo \frac{\lambda_1z^2}{1-\gamma_1z} \cmo \frac{\lambda_2z^2}{1-\gamma_2z} \cmo \dots
\end{equation}
and
\begin{equation} \label{fracbn}
   \sum_ {n\geq0} c_n z^n = 
   \frac{1}{(1+z)(1+y^2z)-\gamma_0z} \cmo \frac{\lambda_1z^2}{(1+z)(1+y^2z)-\gamma_1z} \cmo 
   \frac{\lambda_2z^2}{(1+z)(1+y^2z)-\gamma_2z} \cmo \dots
\end{equation}
Then we have:
\[
   b_n = \sum_{k=0}^n  \left( \sum\limits_{j=0}^{n-k} y^{2j}
    \Big( \tbinom{n}{j}\tbinom{n}{j+k} -
   \tbinom{n}{j-1}\tbinom{n}{j+k+1}\Big) \right)   c_k.
\]
\end{lem}

\begin{proof}
Let $f(z)$ and $g(z)$ respectively denote the generating functions of $(b_n) _{n\ge0}$ and $(c_n) _{n\ge0}$ as in Equations~\eqref{fracan}
and \eqref{fracbn}. 
Divide by $(1+z)(1+y^2z)$ the numerator and denominator of each fraction in \eqref{fracbn}, 
this gives an equivalence of continued fractions so that:
\[
   zg(z) =  \tfrac{z}{(1+z)(1+y^2z)} f\left( \tfrac{z}{(1+z)(1+y^2z)} \right).
\]
It follows that 
\begin{equation} \label{inv1}
  zf(z) = C(z)g(C(z))
\end{equation}
where $C(z)$ is the compositional inverse of $\frac{z}{(1+z)(1+y^2z)}$.
It remains to show that 
\begin{equation} \label{inv2}
  C(z)^{k+1} = \sum_{n\geq0}     \left( \sum\limits_{j=0}^{n-k} y^{2j}
    \Big( \tbinom{n}{j}\tbinom{n}{j+k} -
   \tbinom{n}{j-1}\tbinom{n}{j+k+1}\Big) \right) z^{n+1}.
\end{equation}
Indeed, the lemma follows from taking the coefficient of $z^{n+1}$ in both sides 
of \eqref{inv1} and using \eqref{inv2} to simplify the right-hand side.
Showing \eqref{inv2} can be conveniently done using Lagrange inversion.
For example, with \cite[p.148, Theorem A]{comtet}, we obtain
\[
  [z^{n+1}] C(z)^{k+1} = \frac{k+1}{n+1} [z^{n-k}] \big((1+z)(1+y^2z)\big)^{n+1}
  = \frac{k+1}{n+1} \sum_{j=0}^{n-k} y^{2j} \binom{n+1}{j}\binom{n+1}{n-k-j},
\]
and it is straightforward to check that
\[
   \tfrac{k+1}{n+1} \tbinom{n+1}{j}\tbinom{n+1}{n-k-j} = \tbinom{n}{j}\tbinom{n}{j+k} -
   \tbinom{n}{j-1}\tbinom{n}{j+k+1}.
\]
This completes the proof.
\end{proof}

We can apply this lemma with $b_n= (1-q)^n B_n(y,1,q)$. The continued fraction expansion of 
$\sum b_nz^n$ is immediately obtained from the one of $\sum B_n(y,t,q)z^n$ obtained in Theorem~\ref{bnfrac}. 
More precisely, it is exactly Equation~\eqref{fracan} with the values 
\begin{equation} \label{valuesgl} \begin{split}
\gamma_h  &= y^2(1-q^{h+1}) + (1-q^h) + yq^h(2-q^{h+1}-q^h),  \\
\lambda_h &= y(1-q^h)^2(y+q^{h-1})(1+yq^h).
\end{split}\end{equation}
The theorem is now a consequence of another lemma which gives the value of $c_k$.

\begin{lem}
With $\gamma_h$ and $\lambda_h$ as in \eqref{valuesgl}, we have:
\begin{align*}
   \frac{1}{(1+z)(1+y^2z)-\gamma_0z} \cmo \frac{\lambda_1z^2}{(1+z)(1+y^2z)-\gamma_1z} \cmo 
   \frac{\lambda_2z^2}{(1+z)(1+y^2z)-\gamma_2z} \cmo \dots \\
  = \sum_{k\geq0} z^k  \sum\limits_{j=0}^k y^{k-j} (-1)^j  \sum_{i=0}^{2j} y^i  q^{i(2j-i+1)/2}.
\end{align*}
\end{lem}

\begin{proof}
After multiplying by $1-yz$, this is equivalent to
\begin{equation} \label{eqmul} \begin{split}
   \frac{1-yz}{(1+z)(1+y^2z)-\gamma_0z} \cmo \frac{\lambda_1z^2}{(1+z)(1+y^2z)-\gamma_1z} \cmo 
   \frac{\lambda_2z^2}{(1+z)(1+y^2z)-\gamma_2z} \cmo \dots \\
  = \sum_{j\geq0} (-z)^j  \sum_{i=0}^{2j} y^i  q^{i(2j-i+1)/2}.
\end{split}\end{equation}
A continued fraction expansion of the right-hand side is obtained in the work of the second and third authors \cite{JK}.
More precisely, the substitution $(z,y,q)\mapsto (-yz,yq^{\frac 12},q^{\frac 12})$ in \cite[Theorem 7.1]{JK} gives
\begin{align} \label{eqmul2}
\sum_{j\geq0} (-z)^j  \sum_{i=0}^{2j} y^i  q^{i(2j-i+1)/2} =
\frac 1{1-yz} \cmp \frac{d_1z}{1-yz} \cmp \frac{d_2z}{1-yz} \cmp \frac{d_3z}{1-yz} \cmp \cdots
\end{align}
where  $d_{2h+1}=(1+yq^{h+1})(y+q^h)$ and $d_{2h}=y(1-q^h)^2$. It is immediate to check the relations
\begin{align} \label{relglc}
 \lambda_h = d_{2h-1}d_{2h}, \qquad \gamma_h = (1+y)^2 -d_{2h}-d_{2h+1}
\end{align}
(with the convention $d_0=0$). 
It remains to identify the left-hand side of \eqref{eqmul} with the right-hand side of \eqref{eqmul2}.
The end of the proof follows from the lemma below.
\end{proof}

\begin{lem}
 Let $(\gamma_h)_{h\geq0}$, $(\lambda_h)_{h\geq1}$ and $(d_h)_{h\geq0}$ be three sequences satisfying \eqref{relglc} with $d_0=0$. 
Then we have:
\begin{equation} \label{idfrac}  \begin{split}
   \frac{1-yz}{(1+z)(1+y^2z)-\gamma_0z} \cmo \frac{\lambda_1z^2}{(1+z)(1+y^2z)-\gamma_1z} \cmo 
   \frac{\lambda_2z^2}{(1+z)(1+y^2z)-\gamma_2z} \cmo \dots \\
 =  \frac 1{1-yz} \cmp \frac{d_1z}{1-yz} \cmp \frac{d_2z}{1-yz} \cmp \frac{d_3z}{1-yz} \cmp \cdots.
\end{split}\end{equation}
\end{lem}

\begin{proof}
In the left-hand side, divide the numerator and denominator of each fraction by $1-yz$. Using that 
\[
   \frac{(1+z)(1+y^2z)}{1-yz} = 1-yz + \frac{z(1+y)^2}{1-yz},
\]
the left-hand side of \eqref{idfrac} is
\begin{equation*}
 \frac{1}{1-yz+((1+y)^2-\gamma_0)\frac{z}{1-yz}} \cmo \frac{\lambda_1 \big(\frac{z}{1-yz}\big)^2 }{1-yz+((1+y)^2-\gamma_1)\frac{z}{1-yz}}
 \cmo \cdots,
\end{equation*}
hence it is equal to
\begin{equation} \label{schr2}
 \frac{1}{1-yz+ d_1\frac{z}{1-yz}} \cmo \frac{d_1d_2 \big(\frac{z}{1-yz}\big)^2 }{1-yz+(d_2+d_3)\frac{z}{1-yz}}
 \cmo \frac{d_3d_4 \big(\frac{z}{1-yz}\big)^2 }{1-yz+(d_4+d_5)\frac{z}{1-yz}} \cmo \cdots.
\end{equation}
This can be shown to be equal to the right-hand side of \eqref{idfrac}, using the combinatorics of weighted 
Schr\"oder paths. A {\it Schr\"oder path} (of length $2n$) is a path from $(0,0)$ to $(2n,0)$ in $\mathbb{N}^2$ with steps
$(1,1)$, $(1,-1)$ and $(2,0)$, respectively denoted by $\nearrow$, $\searrow$, $\longrightarrow$. We set that
each step $\longrightarrow$ has a weight $y$, each step $\nearrow$ from height $h-1$ to $h$ has weight $-d_h$. Then
by standard methods, the weighted generating function of all Schr\"oder paths is the continued fraction in the right-hand
side of \eqref{idfrac}. By counting differently, we can obtain \eqref{schr2}. The idea is to split each Schr\"oder path into 
some subpaths, by putting a splitting point each time the path arrives at even height. 
This way, we can see a Schr\"oder path as an ordered sequence of:
\begin{itemize}
 \item A subsequence $\nearrow\longrightarrow \dots \longrightarrow \searrow$
   starting at height $2h$ whose generating function
       is $-d_{2h+1}\frac{z}{1-yz}$.
 \item A subsequence $\searrow\longrightarrow \dots \longrightarrow \nearrow$ starting at height $2h$ whose generating function
       is $-d_{2h}\frac{z}{1-yz}$.
 \item A subsequence $\nearrow\longrightarrow \dots \longrightarrow \nearrow$ starting at height $2h$ whose generating function
       is $d_{2h+1}d_{2h+2}\frac{z}{1-yz}$.
 \item A subsequence $\searrow\longrightarrow \dots \longrightarrow \searrow$ starting at height $2h$ whose generating function
       is $\frac{z}{1-yz}$.
 \item A step $\longrightarrow$ at height $2h$ whose generating function
       is $yz$.
\end{itemize}
The rules according to which these subsequences can be put together is conveniently encoded in a continued fraction so that we 
obtain exactly \eqref{schr2}. More precisely, let $F_h$ be the generating function of Schr\"oder paths from height $2h$ to $2h$
and staying at height $\geq 2h-1$. We have :
\begin{equation} \label{relheight}
  F_h = \frac{1}{1-yz+(d_{2h}+d_{2h+1})\frac{z}{1+yz}-  d_{2h+1}d_{2h+2}(\frac{z}{1-yz})^2F_{h+1}}.
\end{equation}
Indeed, we can see a Schr\"oder paths from height $2h$ to $2h$ and staying at height $\geq 2h-1$ as an ordered sequence of:
\begin{itemize}
 \item steps $\longrightarrow$ at height $2h$ whose generating function
       is $yz$,
 \item subsequences $\nearrow\longrightarrow \dots \longrightarrow \searrow$ starting at height $2h$ whose generating function
       is $-d_{2h+1}\frac{z}{1-yz}$,
 \item subsequences $\searrow\longrightarrow \dots \longrightarrow \nearrow$ starting at height $2h$ whose generating function
       is $-d_{2h}\frac{z}{1-yz}$,
 \item subsequences $\nearrow\longrightarrow \dots \longrightarrow \nearrow  P \searrow\longrightarrow \dots \longrightarrow \searrow $
       where $P$ is a path from height $2h+2$ to $2h+2$ staying above height $2h+1$ whose generating function
       is $d_{2h+1}d_{2h+2}(\frac{z}{1-yz})^2F_{h+1}$.
\end{itemize}
Hence we obtain \eqref{relheight}.
Using \eqref{relheight} for successive values of $h$, we obtain that $F_0$ is the continued fraction in \eqref{schr2}. This 
completes the proof.
\end{proof}

Let us examine the case $q=0$ in the formula \eqref{formulean1}. We obtain:
\[
   B_n(y,1,0) = \sum_{k=0}^n (-1)^k \sum_{i=0}^{2n-2k} y^i \tbinom{n}{k+\lceil i/2 \rceil} \tbinom{n}{\lfloor i/2 \rfloor} 
   = \sum_{i=0}^{2n} y^i \tbinom{n}{\lfloor i/2 \rfloor} \sum_{k=0}^{\lfloor n-\frac{i}2 \rfloor} (-1)^k \tbinom{n}{k+\lfloor i/2 \rfloor}.
\]
The alternating sum of binomial coefficients is itself a binomial coefficient,
and we obtain:
\[
  B_n(y,1,0) = \sum_{i=1}^{2n} y^i \binom{n}{\lfloor i/2 \rfloor} \binom{n-1}{\lceil i/2 \rceil -1}.
\]
We thus obtain, as mentioned in Theorem~\ref{qeulrianb}, that the value at $q=0$ of the $q$-Eulerian numbers of type $B$ are the Narayana numbers of type $B$:
\[
  E_{n,k}^B(0) = [y^{2k}]B_n(y,1,0) + [y^{2k+1}]B_n(y,1,0) =  \binom{n}{k} \binom{n-1}{k-1} + \binom{n}{k}\binom{n-1}{k} = \binom{n}{k}^2.
\]

\begin{remark}
  One can prove the identity $E_{n,k}^B(0) = \binom{n}{k}^2$ bijectively as
  follows.  Considering $\pi\in B_n$ as a permutation on $[\pm n] = \{\pm 1,
  \pm2, \dots,\pm n\}$, define $f(\pi)$ to be the partition of $[\pm n]$
  obtained by making cycles of $\pi$ into blocks. It is not difficult to show
  that the map $f$ is a bijection from the set of $\pi\in B_n$ with
  $\cro(\pi)=0$ to the set of type $B$ noncrossing partitions of $[\pm n]$ such
  that if $\floor{\fwex(\pi)/2}=k$ then $f(\pi)$ has $2k$ nonzero blocks.  It is
  well known that $\binom{n}{k}^2$ is the number of type $B$ noncrossing
  partitions of $[\pm n]$ with $2k$ nonzero blocks, see \cite{Reiner1997}.
\end{remark}

\section{Open problems}
\label{conc}

We conclude this paper by a list of open problems.\\

\noindent{\bf Problem 1.} Since the introduction of permutation tableaux in \cite{Williams2005,Steingrimsson2007}, several variants have been defined
\cite{ABD,ABN2010,Vie2}. A nice feature of these variants is that the permutation statistics arise naturally,
from a recursive construction of the tableaux via an insertion algorithm \cite{ABN2010} . The type $B$ version of these 
tableaux can be defined with the condition of being conjugate-symmetric. A natural question is to check
whether the insertion algorithm  can be used to recover some of our results.\\

\noindent{\bf Problem 2.} One key feature of our new $q$-Eulerian polynomials of type $B$ is their symmetry, {\it i.e.} we have
$B_{n,k}^*(t,q)=B_{n,2n+1-k}^*(t,q)$. We prove this symmetry using the pignose diagram of a signed permutation.
It would be interesting to show this symmetry using the permutation tableaux of type $B$.\\

\noindent{\bf Problem 3.} We have defined alignments in Section~\ref{cro_ali}
and showed that for a signed permutation $\pi$ with $\fwex(\pi)=k$, we have
$2\cro(\pi)+\al(\pi)=n^2-2n+k$. A similar identity exists for the type $A$, see
Proposition~\ref{prop:corteel_intro}, and can be shown on permutations or
directly on permutation tableaux. It would be elegant to show our identity
directly on the permutation
tableaux of type $B$. \\

\noindent{\bf Problem 4.} A  notion that is closely related to alignments and in some sense dual to the 
crossing, is the one of {\it nesting} \cite{Corteel2007}. When we introduce a parameter $p$ counting the number
of crossings in permutations, there are continued fractions containing $p,q$-integers rather than the 
$q$-integers, see  \cite{Corteel2007,SZ}. A definition of nestings in signed permutations have been given 
by Hamdi \cite{hamdi}. It would be interesting to check if our results can be generalized to take into 
account these nestings. \\

\noindent{\bf Problem 5.} In the last section, we have obtained a formula for $B_n(y,1,q)$. We can ask if there is a more general formula
for $B_n(y,t,q)$, but it seems that the present methods do not generalize in this case. \\

\noindent{\bf Problem 6.} Recently Kim and Stanton \cite{kimstanton} gave a
combinatorial proof of the formula \eqref{eq:y0q} for $B_n(y,0,q)$, which is a
generating function for type $A$ permutation tableaux. It is worth asking
whether this combinatorial approach can be generalized for $B_n(y,1,q)$ and
possibly $B_n(y,t,q)$.

\section*{Acknowledgements.} We thank Philippe Nadeau and Lauren Williams for
constructive discussions during the elaboration of this work.

\bibliographystyle{plain}

\end{document}